\documentclass[12pt]{amsart}

\usepackage{amsthm}
\usepackage{amssymb}
\usepackage{amsfonts}
\usepackage{amsmath}
\usepackage[alphabetic,initials]{amsrefs}

\numberwithin{equation}{section}
\theoremstyle{definition}
\newtheorem{definition}[equation]{Definition}
\theoremstyle{definition}
\newtheorem{metadefinition}[equation]{Metadefinition}
\theoremstyle{definition}

\theoremstyle{definition}
\newtheorem{notation}[equation]{Notation}
\theoremstyle{definition}
\newtheorem{construction}[equation]{Construction}
\theoremstyle{remark}
\newtheorem*{remark}{Remark}
\theoremstyle{plain}
\newtheorem{theorem}[equation]{Theorem}
\theoremstyle{plain}
\newtheorem{lemma}[equation]{Lemma}
\theoremstyle{plain}
\newtheorem{proposition}[equation]{Proposition}
\theoremstyle{plain}
\newtheorem{conjecture}[equation]{Conjecture}
\theoremstyle{plain}
\newtheorem{corollary}[equation]{Corollary}
\theoremstyle{plain}
\newtheorem{question}[equation]{Question}
\theoremstyle{plain}
\newtheorem{problem}[equation]{Problem}
\DeclareMathOperator{\Hilb}{Hilb}
\DeclareMathOperator{\shiftme}{Shift}

\DeclareMathOperator{\shift}{Shift}

\DeclareMathOperator{\Span}{span}
\DeclareMathOperator{\shadow}{Shadow}
\DeclareMathOperator{\Shadow}{Shadow}
\DeclareMathOperator{\mdeg}{mdeg}
\DeclareMathOperator{\rk}{rk}
\DeclareMathOperator{\tor}{Tor}
\DeclareMathOperator{\Tor}{Tor}
\DeclareMathOperator{\supp}{supp}
\DeclareMathOperator{\msupp}{Supp}
\DeclareMathOperator{\mSupp}{Supp}
\DeclareMathOperator{\sqfree}{sqfree}
\DeclareMathOperator{\im}{Im}
\DeclareMathOperator{\pol}{pol}
\DeclareMathOperator{\gens}{gens}
\DeclareMathOperator{\init}{in}
\DeclareMathOperator{\Supp}{supp}
\newcommand{\CC}{\mathbb{C}}
\newcommand{\tensor}{\otimes}

\newcommand{\spo}{S^{\mathrm{po}}}
\newcommand{\ipo}{I^{\mathrm{po}}}
\newcommand{\jpo}{J^{\mathrm{po}}}
 \newcommand{\tildephi}{\widetilde{\phi}}
\newcommand{\tildej}{\widetilde{J}}
\newcommand{\bbbk}{\Bbbk}
\begin{document}

\author{Jeff Mermin}
\address{Jeff Mermin\\
  Department of Mathematics\\
  University of Kansas\\ Lawrence, KS 66045\\}
\email{mermin@math.ku.edu}
\author{Satoshi Murai}
\address{Satoshi Murai\\
  Department of Pure and Applied Mathematics\\
  Graduate School of Information Science and Technology\\
  Osaka University\\
  Toyonaka, Osaka, 560-0043, Japan\\
}
\email{s-murai@ist.osaka-u.ac.jp}
\thanks{The first author is supported by an NSF Postdoctoral
  fellowship (award No. DMS-0703625).  The second author is supported
  by  JSPS Research Fellowships for Young Scientists.}

\date{}
\title{The Lex-Plus-Powers Conjecture holds for pure powers}
\maketitle

Abstract: We prove Evans' Lex-Plus-Powers Conjecture for ideals
containing a monomial regular sequence.

\section{Introduction}
Let $S=\bbbk[x_{1},\cdots,x_{n}]$ be the polynomial ring in $n$ variables
over an arbitrary field.  
Fix $r\leq n$ and a nondecreasing sequence of positive integers, $2\leq 
e_{1}\leq e_{2}\leq \cdots\leq e_{r}$, and let
$P=(x_{1}^{e_{1}},\cdots, x_{r}^{e_{r}})$ be the ideal generated by
those powers of the variables.    (If $r\neq n$, it is sometimes
convenient to set
$e_{r+1}=\cdots=e_{n}=\infty$ and $x_{i}^{\infty}=0$.)  

The Hilbert function of a homogeneous ideal of $S$ is a well-studied
and important invariant with applications in many areas, including
Algebraic Geometry, Commutative Algebra, and Combinatorics.  One of
the basic tools in the study of Hilbert functions was provided by
Macaulay \cite{Ma} in 1927:   
every Hilbert function
of a homogeneous 
ideal of $S$ is attained by a lexicographic ideal.  Macaulay's insight
was that the lex ideals, which are defined combinatorially,
are a useful tool in studying the combinatorial invariants of the
polynomial ring.  Later, Macaulay's theorem was extended to many other
rings, including the quotient ring
$S/P$ (due to Clements and Lindstr\"om  \cite{CL}).

Motivated by Macaulay's theorem and applications in Algebraic Geometry,
Eisenbud, Green, and 
Harris made the following conjecture about Hilbert functions
\cites{EGH1,EGH2}:

\begin{conjecture}[Eisenbud, Green, Harris]\label{EGH}
Let $F=(f_{1},\cdots, f_{r})$ be a homogeneous regular sequence, such
that $\deg f_{i}=e_{i}$ for all $i$, and let $I$ be any homogeneous ideal
containing $F$.  Then there is a lex ideal $L$ such that $L+P$
and $I$ have the same Hilbert function.
\end{conjecture}

In recent decades, graded Betti numbers have become an important topic
in Commutative Algebra.  One influential result is due to
 Bigatti \cite{Bi}, Hulett \cite{Hu}, and Pardue \cite{Pa} in
the 1990s.  They showed that the lex ideals of $S$ have maximal graded
Betti numbers among all ideals with a fixed Hilbert function,
providing a sharp upper bound on the graded Betti numbers of a
homogeneous ideal with a given Hilbert function.  Because of the
importance of Bigatti, Hulett, and Pardue's results, 
  similar statements are known or conjectured in many
settings where Macaulay-type theorems hold, including the exterior
algebra and the ring $S/P$
(see for example Aramova-Herzog-Hibi \cites{AHH1, AHH2} and Gasharov-Hibi-Peeva \cite{GHP}).


Inspired by Bigatti, Hulett, and Pardue's results, 
Evans \cite{FR} extended the Eisenbud-Green-Harris conjecture to include a
statement about Betti numbers:

\begin{conjecture}[Evans, The Lex-Plus-Powers Conjecture]
Let $F$, $I$, and $L$ be as in Conjecture \ref{EGH}.
Then for all $i$ and $j$ the graded Betti numbers of $I$ and $L+P$
satisfy $b_{i,j}(L+P)\geq b_{i,j}(I)$.  
\end{conjecture}

Both conjectures are open.  In particular, the Lex-Plus-Powers
Conjecture has been open even if $F$ consists of pure powers of the
variables (i.e., $F=P$).  
The main result of this paper is that 
the Lex-Plus-Powers 
Conjecture holds if $F$ consists of monomials, 
a case in which the Eisenbud-Green-Harris Conjecture is a
straightforward consequence of Clements and Lindstr\"om's Theorem.


For a subset $\tau$ of the variables, put
$\mathbf{x}_{\tau}=\prod_{x_{i}\in\tau} x_{i}^{e_{i}}$.  In
\cite{MPS}, Mermin, Peeva, and Stillman use mapping cones to give a
formula for the Betti 
numbers of a monomial-plus-$P$ ideal in terms of its colon ideals:
If $M$ is a monomial ideal not containing any $x_{i}^{e_{i}}$, then
we have
\begin{equation}\label{thecolons}
b_{i,j}(M+P)=\sum_{\tau}b_{i-|\tau|,j-\deg
  \mathbf{x_{\tau}}}(M:\mathbf{x}_{\tau}).
\end{equation}
Using this formula and the Eliahou-Kervaire resolution \cite{EK},
Murai shows in \cite{Mu1} that the Lex-Plus-Powers Conjecture holds
for Borel-plus-$P$ ideals:

\begin{theorem}[\cite{Mu1}]\label{boreltolex}
Suppose that $B$ is Borel, and let $L$ be a lex ideal such that $L+P$
has
the same Hilbert function as $B+P$.  Then for all $i,j$ we
have $b_{i,j}(L+P)\geq b_{i,j}(B+P)$.  
\end{theorem}

Thus, the Lex-Plus-Powers conjecture would be proved by reduction
to the Borel case:

\begin{question}\label{thequestion}
Let $I$ and $F$ be as in Conjecture \ref{EGH}.  Does
there exist a Borel ideal $B$ such that $B+P$ has the same Hilbert function
as, and larger Betti numbers than, $I$?
\end{question}

In Theorems \ref{lppcharzero} and  \ref{lpp}, we give a positive
answer to Question 
\ref{thequestion} in the case that $F$
consists of monomials.



In section 2, we introduce notation which will be used throughout the
paper.  
In section 3, using a walk on the Hilbert scheme, we prove the
Lex-Plus-Powers Conjecture for 
ideals containing powers of the variables in characteristic zero.
This approach yields a short proof, but does not work in positive
characteristic.  


In sections 4 through 8, we give a characteristic-free proof of the
same result.  While the proof is long, we introduce some new
techniques to study Hilbert functions and Betti numbers of monomial
ideals, including Theorem \ref{keylemma}, a formula for the
multigraded Betti numbers of any monomial ideal.  
Our main tool is a generalization of the combinatorial
``shifting'' operation 
of Erd\"os, Ko, and Rado \cite{EKR}. 

  Shifting is an operation which associates to every simplicial complex
another complex with the same face vector and certain special
properties, called ``shifted''.  (See \cites{AHH2, HM}.)  We
generalize combinatorial shifting to monomial ideals, and show that Betti
numbers are nondecreasing under 
this operation.   We use shifting and compression
(defined in \cite{Me3}) to compare the Betti numbers of an ideal $I$
containing $P$ with those 
of a Borel-plus-$P$ ideal.


We also consider some related problems.
In section 9, we show that the Betti numbers of $I$ are obtained from
those of the lex-plus-powers ideal $L+P$ by
consecutive cancellations.
In section 10, we briefly discuss some open problems.

\medskip

\noindent\textbf{Acknowledgements.}  
The authors thank Chris Francisco, Takayuki Hibi, Craig Huneke, Irena
Peeva, and the 
algebra group at Kansas for 
encouragement and helpful discussions.


\section{Background and Notation}
We recall some notation and results that will be used throughout the paper.

\begin{metadefinition} For a property $(\ast)$ of ideals, and an ideal
  $I$ containing $P$, we say that $I$ is \emph{$(\ast)$-plus-$P$} 
  if there exists an ideal $\hat{I}$
  satisfying $(\ast)$ such that $I=\hat{I}+P$.  In this paper, we will
  consider homogeneous-plus-$P$, lex-plus-$P$, compressed-plus-$P$,
  Borel-plus-$P$, 
  and shifted-plus-$P$ ideals.  
\end{metadefinition}

\begin{notation}  The ring $S$ is graded by setting $\deg x_{i}=1$ for
  all $i$.  All the $S$-modules we consider will inherit a natural
  grading from $S$; if $M$ is a graded module we write $M_{d}$ for the
  $\bbbk$-vector subspace spanned by the homogeneous forms of
  degree $d$ in 
  $M$.  We denote shifts in the grading in the usual way; that is,
  $M(-d)$ is the module isomorphic to $M$ but with all degrees
  increased by $d$, so that, as vector spaces, $M(-d)_{e}=M_{e-d}$.  
\end{notation}

\begin{definition} We will use both the graded \emph{lexicographic}
  and \emph{reverse lexicographic} monomial orderings.  Let $u$ and
  $v$ be monomials of the same 
  degree, and write $u=x_{1}^{e_{1}}x_{2}^{e_{2}}\cdots x_{n}^{e_{n}}$
  and $v=x_{1}^{f_{1}}x_{2}^{f_{2}}\cdots x_{n}^{f_{n}}$.  We say that
  $u$ is \emph{greater than} $v$ \emph{with respect to the
  lexicographic order}, or $u>_{\rm{lex}}v$, if there exists an $i$
  such that $e_{i}>f_{i}$ and $e_{j}=f_{j}$ for all $j<i$.  We say that
  $u$ is \emph{greater than} $v$ \emph{with respect to the
  reverse lexicographic order}, or $u>_{\mathrm{rev}}v$, if there
  exists an $i$ 
  such that $e_{i}<f_{i}$ and $e_{j}=f_{j}$ for all $j>i$.
\end{definition}

\begin{definition}  We say that a monomial ideal $L\subset S$ is
  \emph{lex} or \emph{lexicographic} if, for all degrees $d$, the
  vector space $L_{d}$ is 
  generated by an initial 
  segment in the lexicographic order.  That is, if $u$ and $v$ are monomials of
  the same degree such that
  $u<_{\mathrm{lex}} v$ and $u\in L$, then we must have $v\in
  L$ as well. 
\end{definition}

\begin{definition} We can use these orderings to compare monomial
  ideals as well.  Let $\mathcal{I}=\{u_{1},\cdots,u_{s}\}$ and
  $\mathcal{J}=\{v_{1},\cdots, v_{s}\}$ be sets of degree $d$
  monomials, each ordered reverse lexicographically (so
  $u_{i}>_{\mathrm{rev}}u_{j}$ and $v_{i}>_{\mathrm{rev}}v_{j}$
  whenever $i<j$).  Then we say that $\mathcal{I}$ \emph{is
  reverse lexicographically greater than} $\mathcal{J}$,
  $\mathcal{I}>_{\mathrm{rev}}\mathcal{J}$, if there exists an $i$
  such that $u_{i}>_{\mathrm{rev}}v_{i}$ and $u_{j}=v_{j}$ for all
  $j<i$.  For monomial ideals $I\neq J$ having the same Hilbert
  function, and for a degree $d$, let
  $\{I_{d}\}$ and $\{J_{d}\}$ be the sets of degree $d$ 
  monomials in $I$ and $J$, respectively.  We say that $I$ \emph{is
  reverse lexicographically greater than} $J$ if, for all $d$,
  $I_{d}=J_{d}$ or $\{I_{d}\}$ is
  reverse lexicographically greater than $\{J_{d}\}$.  
\end{definition}

\begin{lemma}
\label{2no1}
Let $\mathcal{I}=\{u_1,\dots,u_t\}$
and $\mathcal{J}=\{v_1,\dots,v_t\}$ be sets of monomials, all with the
same degree.
If $u_k \geq_{\mathrm{rev}} v_k$ for all $k$, then
$\mathcal{I}$ is reverse lexicographically greater than or equal to
$\mathcal{J}$. 
\end{lemma}
\begin{proof}  We use induction on $t$.  If $t=1$, the statement is
  immediate.  Otherwise, let $u_{p}$ and $v_{q}$ be the smallest
  elements of $\mathcal{I}$ and $\mathcal{J}$, respectively, with
  respect to the reverse lex order.  Then, by assumption, we have
  $u_{q}\geq_{\mathrm{rev}} u_{p}\geq_{\mathrm{rev}} v_{p}
  \geq_{\mathrm{rev}}v_{q}$, so we can apply the inductive hypothesis
  to get that $\mathcal{I}\smallsetminus \{u_{p}\}$ is
  reverse lexicographically greater than or equal to
  $\mathcal{J}\smallsetminus 
  \{v_{q}\}$.  Since $u_{p}$ and $v_{q}$ are the smallest elements of
  $\mathcal{I}$ and $\mathcal{J}$, it follows that $\mathcal{I}$ is
  reverse lexicographically greater than or equal to $\mathcal{J}$ as desired.
\end{proof}

Term orders allow us to associate to any ideal of $S$ a monomial
ideal, called its initial ideal.  In this paper we consider only
reverse lexicographic initial ideals, but the definition below works
with any term order.

\begin{definition} 
For a (homogeneous) polynomial $g$, write $g=\sum
  a_{m}m$ with 
  $a_{m}\in \bbbk$ and $m$ ranging over the monomials.  The
  \emph{initial monomial} of $g$, $\init_{\mathrm{rev}}(g)$, is the
  maximal $m$ in the reverse lexicographic order such that $a_{m}$ is
  nonzero.  For an (homogeneous) ideal $I$, the \emph{initial ideal}
  of $I$ is the monomial ideal generated
  by the initial monomials of every form in $I$,
  $\init_{\mathrm{rev}}(I)=(\init_{\mathrm{rev}}(g):g\in I)$.  It is
  well-known that $\init_{\mathrm{rev}}(I)$ has the same Hilbert
  function as $I$ and larger graded Betti numbers.
\end{definition}

\begin{definition} For a graded module $M$, the \emph{Hilbert function} of
  $M$ assigns to each degree $d$ the dimension of the vector space
  $M_{d}$.  We write $\Hilb(M)(d)=\dim_{\bbbk}M_{d}$.  
\end{definition}

\begin{definition}A \emph{free resolution} of the graded module $M$ is
  an exact sequence 
\[
\mathbb{F}:\cdots\to F_{1}\to F_{0}\to M\to 0
\]
such that each $F_{i}$ is a free $S$-module.  We say that $\mathbb{F}$
is the \emph{minimal free resolution} of $M$ if each $F_{i}$ has
minimum possible rank.  Equivalently, $\mathbb{F}$ is minimal if, for
all $i$, the nonzero entries of the matrix associated to the map
$d_{i}:F_{i}\to F_{i-1}$ are contained in the homogeneous maximal
ideal, $(x_{1},\cdots, x_{n})$.  Up to an isomorphism of complices,
every finitely generated module has a unique minimal free resolution.
\end{definition}

\begin{definition}  If $\mathbb{F}$ is the minimal free resolution of
  $M$, the \emph{Betti numbers} of $M$ are given by $b_{i}(M)=\rk
  F_{i}$.  If we decompose the $F_{i}$ as graded free modules,
  $F_{i}=\displaystyle\bigoplus_{j\in \mathbb{Z}} S(-j)^{b_{i,j}}$, then the $b_{i,j}$
  are the \emph{graded Betti numbers} of $M$.  
\end{definition}


\begin{definition} A monomial ideal $I$ is \emph{Borel} or
  \emph{$0$-Borel-fixed} if it satisfies the property:
\begin{center}
If $fx_{j}\in I$ and $i<j$, then
$fx_{i}\in I$.
\end{center}
\end{definition}

Borel ideals are important because they occur (in characteristic zero)
as generic initial ideals \cites{BS,Ga}.  They are combinatorially useful
because they are minimally resolved by the Eliahou-Kervaire
resolution \cite{EK}, which gives explicit formulas for their Betti
numbers.  
Borel ideals can also be attained via a characteristic-free technique
called compression.

\begin{definition}
Fix a subset $\mathcal{A}\subset \{x_{1},\cdots, x_{n}\}$.  Any
monomial ideal $I$ decomposes (as a $\bbbk$-vector space) into a direct
sum over monomials \allowbreak $f\in \bbbk[\{x_{1},\cdots,
  x_{n}\}\smallsetminus\nobreak\mathcal{A}]$, 
$\displaystyle I=\bigoplus_{f}fV_{f}$.  Each $V_{f}$ is an ideal of
$\bbbk[\mathcal{A}]$.  If the $V_{f}$ are all lex ideals in
$\bbbk[\mathcal{A}]$, we say that $I$ is
\emph{$\mathcal{A}$-compressed}.   

Set $W_{f}$ equal to the lex ideal of $\bbbk[\mathcal{A}]$ having the same
Hilbert function as $V_{f}$.  We say that $J=\bigoplus fW_{f}$ is the
\emph{$\mathcal{A}$-compression} of $I$.
\end{definition}

Compression and compressed ideals have been used by Macaulay and
others \cites{CL,Ma,Me1,Me3,MP1,MP2,MPS,HM} to study Hilbert functions and
Betti numbers.
In \cite{Me3}, Mermin proves the following:

\begin{theorem}[\cite{Me3}]\label{compressionbetti}
Let $N$ be a monomial ideal, and let $T$ be the
 $\mathcal{A}$-compression of $N$.  Then:
\begin{itemize}
\item[(i)] $T$ is an ideal.
\item[(ii)] $N$ and $T$ have the same Hilbert function.
\item[(iii)] $b_{i,j}(T)\geq b_{i,j}(N)$ for all $i$ and $j$.
\item[(iv)] $N$ is Borel if and only if $N$ is
  $\{x_{i},x_{j}\}$-compressed for all $i$ and $j$.
\end{itemize}
\end{theorem}

\begin{definition}[\textbf{Polarization}]
For ease of notation, we define a simplified version of polarization.
For a fuller version of the theory, see e.g.\ \cite{Ei2}*{Exercise 3.24}.
Fix a variable $b=x_{k}$.  For a monomial $u=\prod x_{i}^{f_{i}}$, the
$b^{\mathrm{th}}$ \emph{polarization} of $u$ is
\[
\pol_{b}(u)=\displaystyle\left(\prod_{x_{i}\neq
  b}x_{i}^{f_{i}}\right)(bc_{1}\cdots c_{f_{k}-1}),
\]
 where the $c_{i}$ are new variables (and $\pol_{b}(u)=u$ if $b$ does
 not divide $u$).  

Let $s$ be sufficiently large (e.g., the largest power
of $b$ occuring in any generator of any ideal under consideration), and set
$\spo=\bbbk[x_{1},\cdots,x_{n},c_{1},\cdots, c_{s-1}]$.  (We order the
variables so
that $x_{n}>_{\mathrm{rev}}c_{k}$ for all $k$.)
For a
monomial ideal $I$, set $\gens(I)$ equal to the (unique) set of
minimal monomial generators of $I$.  Then for $u\in
\gens(I)$, we have $\pol_{b}(u)\in \spo$.  The \emph{polarization} of
$I$ is the ideal $\ipo$ of $\spo$ generated by these monomials, 
\[
\ipo=\left(\pol_{b}(u):u\in\gens(I)\right).
\]
\end{definition}

\noindent A monomial ideal $I\in S$ is naturally associated to two ideals of
$\spo$, namely $\ipo$ and $I\spo$.  We have the following:

\begin{proposition}\label{polarlemma}\  
\begin{itemize}
\item[(i)] For all $i$ and $j$,
  $b_{i,j}(I)=b_{i,j}(\ipo)=b_{i,j}(I\spo)$.  (Here, $b_{i,j}(\ipo)$
  and $b_{i,j}(I\spo)$ are taken over $\spo$.)
\item[(ii)] Let $I$ and $J$ be monomial ideals of $S$.  Then $\ipo$
  and $I\spo$ have the same Hilbert function, and $I\spo$ and $J\spo$ have
  the same Hilbert function if and only if $I$ and $J$ have the same
  Hilbert function. 
\item[(iii)] Let $u\in I$ be a monomial of $S$ such that
  $\pol_{b}(u)\in \spo$.  Then $\pol_{b}(u)\in
  \ipo$.
\end{itemize}
\end{proposition}
\begin{proof}
(i) is \cite{BH}*{Lemma 4.2.16}, and (ii) is immediate from (i) and
  the formula
\[
\Hilb(I)(d)=\sum_{i,j} \left((-1)^{i}b_{i,j}(I)\Hilb(S)(d-j)\right).
\]
For (iii), observe that $\pol_{b}(v)$ divides $\pol_{b}(u)$ whenever
$v$ divides $u$.
\end{proof}

\section{The proof in characteristic zero}

In this section we prove the following:

\begin{theorem}\label{lppcharzero} Let $\bbbk$ have characteristic zero, and
  let $F=(f_{1},\cdots, f_{r})$ be a
  regular sequence of monomials, in degrees $e_{1}\leq\cdots\leq
  e_{r}$. Set $P=(x_{1}^{e_{1}},\cdots, x_{r}^{e_{r}})$.  If $I$ is
  any ideal containing $F$, then there exists a lex-plus-$P$ ideal $L$
  such that $I$ and $L$ have the same Hilbert function and
  $b_{i,j}(L)\geq b_{i,j}(I)$ for all $i$ and $j$.  
\end{theorem}

Throughout the section, $F=(f_{1},\cdots, f_{r})$ will be a regular
sequence of monomials in degrees $e_{1},\cdots, e_{r}$, and $P$ will
be the pure powers in these degrees, $P=(x_{1}^{e_{1}}, \cdots, x_{r}^{e_{r}})$.
First, we reduce to the case that $I$ is monomial-plus-$P$.

\begin{lemma}\label{anymonregseq}
Let $I$ be a homogeneous ideal containing $F$.  Then
  there exists a monomial ideal $J$ containing $P$ such that $I$ and
  $J$ have the same Hilbert function and $b_{i,j}(J)\geq
  b_{i,j}(I)$ for all $i$ and $j$.
\end{lemma}
\begin{proof}
For any monomial $u$ of $S$,
we set $\supp(u) = \{ x_k:x_k \mbox{ divides $u$}\}$.
Since $f_{1},\cdots, f_{r}$ is a regular sequence, we have
$\supp(f_{i})\cap \supp(f_{j})=\varnothing$ for all $i\neq j$.  After
reordering the variables if necessary, we may assume $x_{i}\in
\supp(f_{i})$.  

Write $\supp(f_{1})=\{ x_{i_{1}},\dots,x_{i_{t}}\}$.
We may assume $i_{1} =1$.  Consider the automorphism $\phi$ of $S$
given by $\phi(x_{k})=x_{k}$ for $x_{k}=x_{1}$ or $x_{k}\not\in\supp(f_{1})$
and
$\phi(x_{k})=x_{1}+x_{k}$ for
$x_{k}\in\supp(f_{1})\smallsetminus\{x_{1}\}$. 
We have $\phi(f_{k})=f_{k}$ for $k\neq 1$, and we can write
$\phi(f_{1})=x_{1}^{e_{1}}+g$ for some polynomial $g$. 
 Set $I'=\init_{\mathrm{rev}}(\phi(I))$.
Then $I'$ contains $(x_{1}^{e_{1}}, f_{2}, \cdots,
f_{r})$, has the same Hilbert function as $I$, 
 and
$b_{i,j}(I')\geq b_{i,j}(I)$. 
Repeating
this procedure for each $f_{k}$ yields an ideal $J$ with the desired
properties.  
\end{proof}

We remark that the proof of Lemma \ref{anymonregseq} is
characteristic-free.  However, for the rest of the section, we will
assume that $\bbbk$ has characteristic zero and that $I$ is a monomial-plus-$P$
ideal.  Since the resolution of a monomial ideal depends only on the
characteristic of the ground field, we may, without loss of
generality, replace $\bbbk$ with any field of characteristic zero.  Thus,
we will assume that $\bbbk=\mathbb{C}$.  

The idea of our proof is similar to that of Pardue \cite{Pa}.
For a monomial-plus-$P$ ideal $I$ which is not Borel-plus-$P$, we
construct
another ideal 
$J$ satisfying:
\begin{itemize}
\item $J$ contains $P$.
\item $\Hilb(J)=\Hilb(I)$.
\item $b_{i,j}(J)\geq b_{i,j}(I)$ for all $i,j$.
\item $J$ is reverse lexicographically greater than $I$.
\end{itemize}
After applying this construction repeatedly, we will obtain a Borel-plus-$P$
ideal and apply Theorem \ref{boreltolex}.

\begin{definition}  Any homogeneous polynomial $f\in S_{d}$ 
  may be written 
  $f=\sum\alpha_{v}v$, where $v$ ranges over the 
  degree $d$ monomials and $\alpha_{v}\in \mathbb{C}$.  
  The \emph{monomial support} of $f$ is the set of monomials
  with nonzero coefficients, $\msupp(f)=\{v:\alpha_{v}\neq 0\}$.  
\end{definition}

\begin{lemma}
\label{tsuika}
Let $d \geq 0$ be an integer and
let $u_1 >_{\mathrm{rev}} \cdots >_{\mathrm{rev}} u_t$ be monomials of
    degree $d$. 
Suppose that $f_1,\dots,f_t$ are $\mathbb{C}$-linearly independent polynomials
of degree $d$
such that $u_k \in \mSupp (f_k)$ for all $k$
and $u_k \not \in \mSupp(f_\ell)$
whenever $ \ell \lneqq k$.
Then
$
\big\{ \init_{\mathrm{rev}}(f) : f \in
\mathrm{span}_\CC\{f_1,\dots,f_t\}\big\}$ 
is reverse lexicographically greater than or equal to
 $\{ u_1,\dots,u_t\}$.
\end{lemma}

\begin{proof}
  We induct on $t$.
If $t=1$ then the statement is obvious.
Otherwise, let $F =\{ \init_{\mathrm{rev}}(f) : f \in
\mathrm{span}_\CC\{f_1,\dots,f_{t-1}\}\}$. 
By induction, we have $F \geq_{\mathrm{rev}} \{u_1,\dots,u_{t-1}\}$.
Let
\[ v \in \left\{ \init_{\mathrm{rev}}(f) : f \in
  \mathrm{span}_\CC\{f_1,\dots,f_t\}\right\} 
\smallsetminus F.
\]
By Lemma \ref{2no1},
it is enough to show that $v \geq_{\mathrm{rev}} u_t$.
By the definition of $v$,
there exist $\alpha_1,\dots,\alpha_t \in \CC$ such that
$
v = \init_{\mathrm{rev}} ( \alpha_1 f_1 + \cdots + \alpha_t f_t).
$
Since $v \not \in F$ we have $\alpha_t \ne 0$, so
$u_t \in \mSupp(\alpha_1 f_1 + \cdots + \alpha_t f_t)$.
Thus, 
$v=\init_{\mathrm{rev}} ( \alpha_1 f_1 + \cdots + \alpha_t f_t) \geq_{\mathrm{rev}}
u_t$ as desired. 
\end{proof}

For the remainder of the section, fix two variables $a>_{\mathrm{rev}}b$.  


\begin{proposition}\label{2no3easy}
 Suppose that $P$ contains no power of $b$, and
  that $I$ is not $\{a,b\}$-compressed-plus-$P$.  Consider the automorphism
  $\phi$ of $S$ given by $\phi(x_{k})=x_{k}$ for $x_{k}\neq b$ and
  $\phi(b)=a-b$.  Set $J=\init_{\mathrm{rev}}(\phi(I))$.  Then:
\begin{itemize}
\item[(i)] $J$ contains $P$.
\item[(ii)] $J$ has the same Hilbert function as $I$.
\item[(iii)] $b_{i,j}(J)\geq b_{i,j}(I)$ for all $i$ and $j$.
\item[(iv)] $J\neq I$.
\item[(v)] $J$ is reverse lexicographically greater than $I$.
\end{itemize}
\end{proposition}
\begin{proof}
(i), (ii), (iii), and (iv) are immediate; we prove (v).  
For any degree $d$, let $\{I_{d}\}$ be the set of degree $d$ monomials
in $I$.  Write $\{I_{d}\}=\{u_{1},\cdots, u_{t}\}$, ordered
reverse lexicographically.  Then
$\{\phi(u_{1}),\cdots, \phi(u_{t})\}$ is a $\mathbb{C}$-basis for
$\phi(I)_{d}$.  Clearly, $u_{k}\in\mSupp(\phi(u_{k}))$ for all $k$ and
$u_{k}\not\in\mSupp(\phi(u_{\ell}))$ for all $\ell\lneqq k$.  Applying
Lemma \ref{tsuika}, $\{J_{d}\}$ is reverse lexicographically greater
than or equal to 
$\{I_{d}\}$.  Since $d$ was arbitrary, it follows that $J$ is
reverse lexicographically greater than $I$, proving (v).
\end{proof}


Next, we consider the case that $P$ contains some power of $b$.

\begin{definition}
Let $e_{b}$ be the smallest power of $b$ appearing in $P$ (i.e.,
$b^{e_{b}}$ is a generator of $P$), and let $\zeta$ be a primitive
$e_{b}^{\ \mathrm{th}}$ root of unity (e.g.,
$\zeta=\cos\frac{2\pi}{e_{b}}+\sqrt{-1}\sin\frac{2\pi}{e_{b}}$). 
Let $\tildephi$ be the autormorphism of $\spo$ given by
$\tildephi(x_{k})=x_{k}$ for $x_{k}\neq b$, $\tildephi(b)=a-b$, and
$\tildephi(c_{k})=a-\zeta^{k}b+c_{k}$ for all $c_{k}$.  Put
$\tildej=\init_{\mathrm{rev}}(\tildephi(\ipo))$.  
\end{definition}

We recall an arithmetic fact about roots of unity:

\begin{lemma}\label{arithmetic}\ 
Let $f=(a-b)(a-\zeta b)\cdots(a-\zeta^{k}b)$.  Then:
\begin{itemize}
\item[(i)] If $k=e_{b}-1$, then $f=a^{k}-b^{k}$.
\item[(ii)] If $k\lneqq e_{b}-1$, then $ab^{k-1}\in\mSupp(f)$.
\end{itemize}
\end{lemma}

\begin{lemma}\label{isnice}  Suppose that $P$ contains some power of
  $b$, and that $I$ is not $\{a,b\}$-compressed-plus-$P$.  Then:
\begin{itemize}
\item[(i)] $\tildej$ contains $P\spo$.
\item[(ii)] $\tildej$ has the same Hilbert function as $I\spo$.
\item[(iii)] $b_{i,j}(\tildej)\geq b_{i,j}(I\spo)$ for all $i$ and $j$.
\item[(iv)] $\tildej\neq I\spo$.
\item[(v)] $\tildej\cap S$ is reverse lexicographically greater than
  $I$.
\item[(vi)] $\tildej=(\tildej\cap S)\spo$.  
\end{itemize}
\end{lemma}

\begin{proof}

To prove (i), it suffices to show that $b^{e_{b}}\in \tildej$.  We have
\begin{align*}
\tildephi\left(\pol_{b}(b^{e_{b}})\right)
&=(a-b)\prod_{k=1}^{e_{b}-1}\Big((a-\zeta^{k}b)+c_{k}\Big)\\  
&=\prod_{k=1}^{e_{b}}\Big(a-\zeta^{k}b\Big) + g\\
&=a^{e_{b}}-b^{e_{b}}+g,
\end{align*}
where every term of the polynomial $g$ is divisible by some $c_{k}$.
In particular, since $a^{e_{b}}\in \ipo$, we have
$\tildephi(a^{e_{b}}-\pol_{b}(b^{e_{b}}))=b^{e_{b}}+g\in\tildephi(\ipo)$,
so $b^{e_{b}}\in \tildej$.  

(ii) and (iii) are immediate from Proposition
  \ref{polarlemma}.
We will prove (iv), (v), and (vi) simultaneously.

 Let $\mathcal{A}=\{a,b, c_{1},\cdots,
  c_{s-1}\}$ be the set consisting of $a$, $b$, and all of the
  $c$-variables, put  $R=\mathbb{C}[\mathcal{A}]$, and consider the
  decomposition $I=\bigoplus 
  fI_{f}$, where $f$ ranges over the monomials of $S$ which are not
  divisible by $a$ or $b$.  Since $\tildephi$ restricts to an
  automorphism of $R$, we get $\tildej=\bigoplus
  f\tildej_{f}=\bigoplus 
  f\left(\init_{\mathrm{rev}}(\tildephi((I_{f})^{\mathrm{po}}))\right)$. It
  suffices to show that
  $\{(\tildej_{f}\cap
  \mathbb{C}[a,b])_{d}\}\geq_{\mathrm{rev}}\{(I_{f})_{d}\}$ for all 
  $d$ (where $\{(I_{f})_{d}\}$ is the set of degree $d$ monomials in
  $I_{f}$, etc.), that $\tildej_{f}\neq I_{f}R$ whenever $I_{f}$
  is not lex-plus-$(b^{e_{b}})$ in $\mathbb{C}[a,b]$, and that
  $\tildej_{f}=(\tildej_{f}\cap \mathbb{C}[a,b])R$.  

  Write $\{(I_{f})_{d}\}=\{a^{p_{1}}b^{q_{1}}, \cdots,
  a^{p_{t}}b^{q_{t}}\}$ in reverse lex order (so
  $q_{1}<\cdots<q_{t}$.)  We have
  \[
  \tildephi(\pol_{b}(a^{p_{k}}b^{q_{k}})) =
  a^{p_{k}}\left(\left(\prod_{\ell=0}^{q_{k}-1}
  (a-\zeta^{\ell}b)\right)+g_{k}\right),
  \]
  where every term of $g_{k}$ is divisible by some $c_{t}$.  
  We have $a^{p_{k}}b^{q_{k}}\in
  \mSupp(\tildephi(\pol_{b}(a^{p_{k}}b^{q_{k}})))$ for all $k$, and
  $a^{p_{k}}b^{q_{k}}\not \in
  \mSupp(\tildephi(\pol_{b}(a^{p_{\ell}}b^{q_{\ell}})))$ for all
  $\ell\lneqq k$.  Set
  \[
  F_{d}=
  \{\init_{\mathrm{rev}}(\tildephi(\pol_{b}(a^{p_{k}}b^{q_{k}}))): 
  a^{p_{k}}b^{q_{k}}\in \{(I_{f})_{d}\}\}.
\]
  Then, by Lemma
  \ref{tsuika}, it follows that
$F_{d}\geq_{\mathrm{rev}}\{(I_{f})_{d}\}$.  In particular,
$F_{d}\subset \mathbb{C}[a,b]$.

  Let $J_{f}$ be the ideal of $\mathbb{C}[a,b]$ generated by all the 
  $F_{d}$, $d\geq 0$.  Immediately we have $\Hilb(J_{f})(d)\geq
  \Hilb(I_{f})(d)$ for all $d$, so 
  it follows that $\Hilb(J_{f}R)(d)\geq
  \Hilb(I_{f}R)(d)=\Hilb(\tildej_{f})(d)$ for all $d$.
  But $J_{f}R\subset \tildej_{f}$, so it must be the case that
  $J_{f}R=\tildej_{f}$ and $J_{f}=\bigoplus \Span_{\mathbb{C}}(F_{d})$.    
  This proves $\tildej_{f}\cap \mathbb{C}[a,b]=J_{f}$, so (vi) holds.
  Also, since $F_{d}\geq_{\mathrm{rev}}\{I_{f}\}_{d}$, it follows that
  $\tildej_{f}\cap \mathbb{C}[a,b]=\bigoplus\Span_{\mathbb{C}}(F_{d})$ is
  reverse lexicographically greater than or equal to $I_{f}$.


  For (iv) and (v), it remains to show that, if $I_{f}$ is not
  lex-plus-$(b^{e_{b}})$, then $\tildej_{f}\neq I_{f}R$.  In this
  case, there exists a degree $d$ and an index $k$ such that
  $q_{k}\lneqq e_{b}$, and $v=a^{p_{k}}b^{q_{k}}\in (I_{f})_{d}$ but
  $u=a^{p_{k}+1}b^{q_{k}-1}\not\in(I_{f})_{d}$.  
  It follows from Lemma \ref{arithmetic} that
  $u\in\mSupp(\tildephi(\pol_{b}(v)))$, but
  $u\not\in\mSupp(\tildephi(\pol_{b}(a^{p_{\ell}}b^{q_{\ell}})))$ for
  any $\ell\lneqq k$.  Thus, by Lemma \ref{tsuika}, $F_{d}$
  is (strictly) reverse lexicographically greater than
  $\{(I_{f})_{d}\}$, and in particular $J_{f}\neq I_{f}$ and
  $\tildej_{f}\neq I_{f}R$.    
\end{proof}


\begin{corollary}\label{2no3hard} Suppose that $P$ contains some power of
  $b$, and that $I$ is not $\{a,b\}$-compressed-plus-$P$.  Set
  $J=\tildej\cap S$.  Then:
\begin{itemize}
\item[(i)] $J$ contains $P$.
\item[(ii)] $J$ has the same Hilbert function as $I$.
\item[(iii)] $b_{i,j}(J)\geq b_{i,j}(I)$ for all $i$ and $j$.
\item[(iv)] $J\neq I$.
\item[(v)] $J$ is reverse lexicographically greater than $I$.
\end{itemize}
\end{corollary}

\begin{proposition}\label{toborel}
 Let $I$ be a monomial ideal containing $P$.  Then
  there exists a Borel-plus-$P$ ideal $B$ such that $B$ has the same
  Hilbert function as $I$, and $b_{i,j}(B)\geq b_{i,j}(I)$ for all
  $i,j$.
\end{proposition}
\begin{proof} If $I$ is not already Borel-plus-$P$, there exist pairs
  of variables $a,b$ such that $I$ is not
  $\{a,b\}$-compressed-plus-$P$.  Choose any such pair.  Define $J$ as
  in Corollary \ref{2no3hard} if $P$ contains some power of $b$, and as in
  Proposition \ref{2no3easy} otherwise. 
  By Corollary \ref{2no3hard} or
  Proposition \ref{2no3easy}, $J$ has the same Hilbert
  function as $I$ and larger Betti numbers.  
Replace $I$ with $J$ and repeat this procedure.  The
  process must terminate 
  since there are finitely many monomial ideals with the same Hilbert
  function, and at each step we are replacing the ideal with a reverse
  lexicographically greater one.  Let $B$ be the resulting ideal.
\end{proof}

Theorem \ref{boreltolex} completes the proof of Theorem
\ref{lppcharzero}.

\begin{proof}[Proof of Theorem \ref{lppcharzero}]
By Lemma \ref{anymonregseq}, we may assume without loss of
generality that $F=P$, that $I$ is a monomial ideal, and that
$\bbbk=\mathbb{C}$.  By Proposition \ref{toborel}, we may assume that $I$
is Borel-plus-$P$.  Thus, the desired inequality holds by Theorem
\ref{boreltolex}.
\end{proof}

\section{Further notation}
For the duration of the paper, $\bbbk$ will be an arbitrary field.
Frequently it will be necessary to slice modules more finely than is
possible with the standard grading.  To this end, we use the
multigraded structure of $S$:

\begin{notation}
We write multidegrees multiplicatively.  That is, we set  $\mdeg
x_{i}=\nobreak x_{i}$ for all $i$, so that the multidegrees are indexed by the
monomials of $S$.  We have $S=\bigoplus S_{m}$, where $m$ ranges over all
the monomials, and $S_{m}$ is the one-dimensional $\bbbk$-vector space
spanned by $\{m\}$.  The modules we consider will all inherit a
multigraded structure from $S$, and shifts in the grading will be
written multiplicatively, so, for monomials $u$ and $v$, we will have
$M(u^{-1})_{v}=M_{u^{-1}v}$ as vector spaces.  
\end{notation}

\begin{remark} Whenever we have a map $\phi:M\to N$ of graded
  (respectively, multigraded) modules, $\phi$ will be \emph{homogeneous
  of degree $0$} (resp., \emph{multihomogeneous of degree $1$}); that is,
  $\phi$ will satisfy $\phi(M_{d})\subset N_{d}$ for all $d$ (resp.,
  $\phi(M_{m})\subset N_{m}$ for all $m$).  Verification of this
  property for each of the maps defined in the paper is
  straightforward, and so will be omitted.
\end{remark}

\begin{definition}
If $\mathbb{F}$ is the minimal free resolution of $M$, and we
decompose the free modules $F_{i}$ as multigraded modules,
$F_{i}=\bigoplus S(m^{-1})^{b_{i,m}}$, we say that the $b_{i,m}$ are
the \emph{multigraded Betti numbers} of $M$.
\end{definition}

Tensoring the resolution $\mathbb{F}$ by $\bbbk$, we get
$b_{i}(M)=\dim_{\bbbk}\tor_{i}(\bbbk,M)$, 
$b_{i,j}(M)=\dim_{\bbbk}\tor_{i}(\bbbk,M)_{j}$, and
$b_{i,m}(M)=\dim_{\bbbk}\tor_{i}(\bbbk,M)_{m}$.

\begin{construction}
Since Tor is balanced, we can compute Betti numbers via a resolution
of $\bbbk$, thus avoiding the more difficult problem of computing a
resolution of $M$.   The minimal resolution of $\bbbk$ is given by the
Koszul complex 
\[
\mathbb{K}:K_{n}\to K_{n-1}\to\cdots \to K_{1}\to S\to \bbbk\to 0.
\]
Each $K_{i}$ is the $i^\mathrm{th}$ exterior power of $K_{1}$; it has
a free basis given by the symbols $e_{\mu}$, where $\mu$ ranges over
the 
squarefree monomials of degree $i$.  The symbol $e_{\mu}$ has degree $i$
and multidegree $\mu$.  If $\mu=x_{j_{1}}\cdots x_{j_{i}}$ with
$j_{1}<\cdots <j_{i}$, we write
$e_{\mu}=e_{x_{j_{1}}}\wedge\cdots\wedge e_{x_{j_{i}}}$.  The differential
is given on this basis by $D(e_{x_{j_{1}}}\wedge\cdots\wedge
e_{x_{j_{i}}})=\sum_{c=1}^{i}
(-1)^{c+1}x_{j_{c}}e_{\frac{\mu}{x_{j_{c}}}}$.  

Thus, the Betti numbers of $M$ can be computed from the homology of
the complex 
\[
\cdots\to M\tensor K_{i+1}\to M\tensor K_{i}\to\cdots.
\]
If $M=I$ is a monomial ideal of $S$, the module $M\tensor K_{i}$ is
the subcomplex of $\mathbb{K}$ generated
(as a $\bbbk$-vector space) by terms of the form $fe_{\mu}$,
where $f\in I$ 
is a monomial and $\mu$ is a squarefree monomial of degree $i$.  The
term $fe_{\mu}$ 
has degree $\deg(f\mu)$ and multidegree $f\mu$.  Its differential is
$D(fe_{\mu})=fD(e_{\mu})$.  

If $M=S/I$ is the quotient by a monomial ideal, then
$\Tor_{i}(S/I,\bbbk)=\Tor_{i-1}(I,\bbbk)$ from the
resolutions of $I$ and 
$S/I$.  We will, without comment, use the homology of
$\mathbb{K}\tensor I$ rather than  that of $\mathbb{K}\tensor S/I$ in
our computations.

\end{construction}

This approach yields a formula for the multigraded Betti numbers of
any monomial ideal.

\begin{definition}
Let $I$ be a monomial ideal, and let $m=\prod x_{\ell}^{e_{\ell}}$ be a
monomial.  Put $\supp(m)=\{x_{\ell}:e_{\ell}\neq 0\}$ and
$\sqrt{m}=\displaystyle\prod_{e_{\ell}\neq 0} 
x_{\ell}$.  The \emph{shadow of $m$ in $I$} is the squarefree monomial
ideal of $\bbbk[\supp(m)]$ given by
\[
\Shadow_{m}(I)=\sqfree\left((I:\frac{m}{\sqrt{m}})\cap 
\bbbk[\supp(m)]\right).
\]
(For a monomial 
ideal $J$, $\sqfree(J)$ is the ideal generated by the squarefree
monomials in $J$.)
\end{definition}

\begin{theorem}\label{keylemma}
Let $I$ be a monomial ideal, and fix a multidegree $m$.  Then, for all
integers $i$, the following numbers are equal:
\begin{itemize}
\item[(i)] $b_{i,m}(I)$.
\item[(ii)] $b_{i,m}(I\cap (\frac{m}{\sqrt{m}}))$.
\item[(iii)] $b_{i,\sqrt{m}}(I:\frac{m}{\sqrt{m}})$.
\item[(iv)] $b_{i,\sqrt{m}}(\Shadow_{m}(I))$.
\end{itemize}
\end{theorem}
Note that (i), (ii), and (iii) are Betti numbers of ideals of $S$,
while (iv) is a Betti number of an ideal of $\bbbk[\supp(m)]$.
This ideal can, however, be treated as an ideal of $S$ without altering its
Betti numbers.  Note also that (iv) is a Betti number of a squarefree
ideal, and can be 
computed with Hochster's formula \cite{Ho}.

\begin{proof}
For a monomial $m$, the multigraded Betti number $b_{i,m}(I)$ is the
$i^{\mathrm{th}}$ homology of the complex of vector spaces
$\left(\mathbb{K}\tensor I\right)_{m}$, which has a $\bbbk$-basis given
by 
\[
\left\{fe_{\mu}:f\in I, f\mu=m, f\mbox{ and }\mu \mbox{ are
  monomials}, \mu\mbox{ is squarefree}\right\}.
\]
Since any $f$ appearing in this basis is contained in
$I\cap(\frac{m}{\sqrt{m}})$, it follows that
$(\mathbb{K}\tensor
I)_{m}=(\mathbb{K} \tensor
(I\cap(\frac{m}{\sqrt{m}})))_{m}$, so (i)
is equal to (ii).  

On the other hand, if $m$ is squarefree, then any $f$ appearing in
this basis is a squarefree monomial of $\bbbk[\Supp(m)]$.  Thus, the
complices $(\mathbb{K}\tensor(I:\frac{m}{\sqrt{m}}))_{\sqrt{m}}$ and 
$\left(\mathbb{K}\tensor\Shadow_{m}(I)\right)_{\sqrt{m}}$ are the same.
Hence (iii)=(iv).

The isomorphism $\frac{m}{\sqrt{m}}\cdot
(I:\frac{m}{\sqrt{m}})=I\cap (\frac{m}{\sqrt{m}})$ gives us
(ii)=(iii), completing the proof.  
\end{proof}

Finally, we recall ``combinatorial shifting'' of squarefree ideals.

\begin{definition} Let $I$ be a squarefree ideal (i.e., $I$ is
  generated by squarefree monomials).
We say that $I$ is \emph{squarefree Borel} or \emph{shifted} if it
satisfies the following 
property:
\begin{quotation}
Let $f$ be a monomial such that $fx_{i}$ and $fx_{j}$ are squarefree,
and suppose $i<j$.  Then $fx_{j}\in I \Rightarrow fx_{i}\in I$.  
\end{quotation}
\end{definition}
\noindent Shifted ideals arise as the Stanley-Reisner ideals of shifted
simplicial complices, and are well-studied in combinatorics.

\begin{definition} Fix two variables $a>_{\mathrm{lex}}b$.  The
  \emph{combinatorial shift} of a squarefree ideal $I$ is the ideal
  $\shift_{a,b}(I)$ generated by:
\[
\shift_{a,b}(I)=\left(\begin{array}{ccc}
f&:&f\in I\\
fa&:&fa\in I \mbox{ or }fb\in I\\
fb&:&fa\in I \mbox{ and }fb\in I\\
fab&:&fab\in I
			\end{array}\right),
\]
where $f$ runs over all the squarefree monomials not divisible by $a$
or $b$.  
\end{definition}

Combinatorial shifting was introduced by Erd\"os, Ko, and Rado \cite{EKR} for
simplicial complexes.  Their definition is equivalent to the one given
above under the Stanley-Reisner correspondence.
The ideal $\shift_{a,b}(I)$ may readily be shown to be a squarefree
ideal having the same 
Hilbert function as $I$, and any squarefree ideal can be transformed
into a shifted ideal by a sequence of combinatorial shifts.
A generalization of this construction to (not necessarily squarefree)
monomial ideals is a major element in our proof of Theorem \ref{lpp}.
In \cite{HM}, Murai and Hibi show that Betti numbers increase under
combinatorial shifting:

\begin{theorem}[\cite{HM}]\label{HM}
Let $I$ be a squarefree ideal, and put $J=\shift_{a,b}(I)$.  Then
$b_{i,j}(S/J)\geq b_{i,j}(S/I)$ for all $i$ and $j$.
\end{theorem}

The proof given in \cite{HM} (which is the inspiration for section 5 of this
paper) is involved.  For the convenience of the reader, and in the
spirit of our proof of Theorem 
\ref{lppcharzero}, we give a shorter proof here:

\begin{proof}
Let $\phi$ be the automorphism of $S$ given by $\phi(b)=a-b$ and
$\phi(x_{k})=x_{k}$ for $x_{k}\neq b$.  Put
$I'=\init_{\mathrm{rev}}(\phi(I))$.  A straightforward computation shows
\[
I'\supseteq \left(\begin{array}{ccc}
f&:&f\in\shift_{a,b}(I)\\
fa&:&fa\in \shift_{a,b}(I)\\
fb&:&fb\in \shift_{a,b}(I)\\
fa^{2}&:&fab\in \shift_{a,b}(I)
			\end{array}\right),
\]
for all squarefree monomials $f$ not divisible by $a$ or $b$.

Define the automorphism $\tildephi$ of $\spo$ by
$\tildephi(c_{k})=b-c_{k}$ for all $c_{k}$
and $\tildephi(x_{\ell})=x_{\ell}$ for all $x_{\ell}$, and set
$\tildej=\init_{\mathrm{rev}}(\tildephi((I')^{\mathrm{po}}))$.  

A straightforward computation gives us  $\tildej\supseteq J\spo$.
Since $\Hilb(I)=\Hilb(I')=\Hilb(J)$, it follows from Proposition
\ref{polarlemma}(ii) that $J\spo$ and
$\tildephi((I')^{\mathrm{po}})$ have the same Hilbert function; hence
$J\spo$ and $\tildej$ have the same Hilbert function.  Thus
$\tildej=J\spo$.  Hence, by Proposition \ref{polarlemma}(i), we have
\[b_{i,j}(S/J)=b_{i,j}(\spo/\jpo)\geq
b_{i,j}(\spo/(I')\spo)=b_{i,j}(S/I')\geq b_{i,j}(S/I).\qedhere\]
\end{proof}

\section{Shifted ideals}
Throughout the rest of the paper, we fix two variables $a$ and $b$,
with $a$ before $b$ in the lex order.  Furthermore $\ell$
(``large'') 
and $s$ (``small'') will always be integers with $\ell\gneqq s\geq 0$, and
$f$ will be a  
monomial not divisible by either $a$ or $b$.  

We begin by generalizing ``shifting'' to arbitrary monomial ideals.

\begin{definition}
Let $I$ be a monomial ideal.  We say that $I$ is
\emph{$(a,b)$-shifted} if, whenever $fa^{s}b^{\ell}\in I$, we have
$fa^{\ell}b^{s}\in I$ as well.  For an integer $t$, we say that $I$ is
\emph{$(a,b,t)$-shifted} if, whenever $fa^{s}b^{\ell+t}\in I$, we have
$fa^{\ell}b^{s+t}\in I$ as well.  Finally, we say that $I$ is
\emph{$(a,b)$-strongly shifted} if $I$ is $(a,b,t)$-shifted for all
nonnegative $t$.  
\end{definition}

\begin{remark}
Suppose that $I$ is a squarefree ideal.  Then $I$ is $(a,b)$-shifted if and
only if $I$ is $\{a,b\}$-squarefree compressed (as defined in
\cites{Me3, MPS}), and shifted if and only if it is $(a,b)$-shifted
for all $a$ and $b$.  
\end{remark}

\begin{definition}
Let $I$ be a monomial ideal.  We define the \emph{$(a,b)$-shift of
  $I$} as the $\bbbk$-vector space  
\[
J=\shiftme_{a,b}(I)=\left<\begin{array}{ccc}
fa^{s}b^{s}&:&fa^{s}b^{s}\in I\\
fa^{\ell}b^{s}&:&fa^{\ell}b^{s}\in I \mbox{ or }fa^{s}b^{\ell}\in I\\
fa^{s}b^{\ell}&:&fa^{\ell}b^{s}\in I \mbox{ and }fa^{s}b^{\ell}\in I
			\end{array}\right>
\]
this basis taken over all $f$ and all pairs $(s,\ell)$ with $s\lneqq\ell$.  

For nonnegative integers $t$, we would like to define the $t^{\mathrm{th}}$
  $(a,b)$-shift of $I$ as
  $\shiftme_{a,b,t}(I)=a^{-t}\shiftme_{a,b}(a^{t}I)$, but it is not
  obvious \emph{a priori} that this even makes sense.  Instead, we
  define the \emph{$t^{\mathrm{th}}$ $(a,b)$-shift of $I$} as
  the $\bbbk$-vector space
\[
J=\shiftme_{a,b,t}(I)=\left<\begin{array}{ccc}
fa^{s}b^{r}&:&fa^{s}b^{r}\in I, r\lneqq t\\
fa^{s}b^{s+t}&:&fa^{s}b^{s+t}\in I\\
fa^{\ell}b^{s+t}&:&fa^{\ell}b^{s+t}\in I \mbox{ or }fa^{s}b^{\ell+t}\in I\\
fa^{s}b^{\ell+t}&:&fa^{\ell}b^{s+t}\in I \mbox{ and }fa^{s}b^{\ell+t}\in I
			\end{array}\right>,
\]
this basis taken over all $f$, all $r\lneqq t$, and all pairs $s\lneqq \ell$.
In Proposition \ref{tshiftisshift}, we will show that this is
equivalent to the desired definition.
\end{definition}

The shifting operation modifies the ideal $I$ by replacing, wherever
possible, monomials of the form $fa^{s}b^{\ell}$ with the (lexicographically
bigger) $fa^{\ell}b^{s}$.  Where this is impossible (because
$fa^{\ell}b^{s}$ is already present), it instead does nothing.  Note
that $\shiftme_{a,b}(I)=\shiftme_{a,b,0}(I)$.

\begin{proposition}\label{shiftmacaulay}
Let $J=\shiftme_{a,b,t}(I)$.  Then:
\begin{itemize}
\item[(i)] $J$ is an ideal.
\item[(ii)] $J$ is $(a,b,t)$-shifted.
\item[(iii)] $J$ has the same Hilbert function as $I$.
\item[(iv)]  $J$ is reverse lexicographically
  greater than or equal to $I$.
\end{itemize}
\end{proposition}

\begin{proof}
(ii), (iii), and (iv) are immediate; we prove (i).

It suffices to show that, for any monomial $m\in J$, we have $ma\in
J$, $mb\in J$, and $mx_{i}\in J$ for any $x_{i}\neq a,b$.  We consider
four cases, depending on the form of $m$.  

Suppose first that $m=fa^{s}b^{r}$ with $r\lneqq t$.  Then $m\in I$,
so we have 
$ma\in I\Rightarrow ma\in J$ and $mx_{i}\in I\Rightarrow mx_{i}\in
J$.  Also, $mb=fa^{s}b^{r+1}\in I$. If $r+1\lneqq t$ this implies
$mb\in J$ immediately; if $r+1=t$ we have $mb=fa^{s}b^{t+0}$, and
$s\geq 0$ gives us $mb\in J$.

Now suppose that $m=fa^{s}b^{s+t}$.  Then $mx_{i}\in I\Rightarrow
mx_{i}\in J$.  Furthermore, $fa^{s+1}b^{s+t}\in I$ and
$fa^{s}b^{s+1+t}\in I$, so these must both be in $J$ as well.

Thirdly, suppose that $m=fa^{\ell}b^{s+t}$.  Then we have
$fa^{\ell}b^{s+t}$ or $fa^{s}b^{\ell+t}$ in   $I$.  It follows that
$mx_{i}\in J$ because $fx_{i}a^{\ell}b^{s+t}\in I$ or
$fx_{i}a^{s}b^{\ell+t}\in I$, that $ma\in J$ because
$fa^{\ell+1}b^{s+t}\in I$ or $fa^{s}b^{\ell+1+t}\in I$, and that
$mb\in J$ because $fa^{\ell}b^{s+1+t}\in I$ or $fa^{s+1}b^{\ell+t}\in
I$.

Finally, suppose that $m=fa^{s}b^{\ell+t}$.  Then we have
$fa^{\ell}b^{s+t}$ and $fa^{s}b^{\ell+t}$ in $I$.  It follows that
$mx_{i}\in J$ because $fx_{i}a^{\ell}b^{s+t}\in I$ and
$fx_{i}a^{s}b^{\ell+t}\in I$, that $ma\in J$ because
$fa^{\ell+1}b^{s+t}\in I$ and $fa^{s}b^{\ell+1+t}\in I$, and that
$mb\in J$ because $fa^{\ell}b^{s+1+t}\in I$ and $fa^{s+1}b^{\ell+t}\in
I$.
\end{proof}

\begin{remark}
For simplicity, let $t=0$ (or make the appropriate changes for
arbitrary $t$).    We could attempt to define a ``pseudograding'' on
$S$ by setting $\mathrm{pdeg}\  m=m$ for a monomial not of the form
$fa^{s}b^{\ell}$, and $\mathrm{pdeg}\
fa^{s}b^{\ell}=fa^{\ell}b^{s}$.  (This is not an actual grading
because $S_{m}S_{n}\not\subseteq S_{mn}$.)  In this pseudograding, $S_{m}$
has dimension $1$ or $2$ for every pseudodegree $m$, and the lex
ideals are precisely the shifted ideals.  Proposition
\ref{shiftmacaulay} states that every pseudo-Hilbert function is
attained by a pseudo-lex ideal, i.e., Macaulay's theorem \cite{Ma}
holds in this setting.  The next natural question is whether the
theorem of Bigatti, Hulett, and Pardue \cites{Bi, Hu, Pa} on Betti
numbers holds as well.  Corollaries \ref{fixedmulti} and
\ref{unfixedmulti} will show that it does.
\end{remark}

\begin{proposition}\label{tshiftisshift}
 Let $J=\shiftme_{a,b,t}(I)$.  Then
  $a^{t}J=\shiftme_{a,b,0}(a^{t}I)$.  
\end{proposition}
\begin{proof}
As vector spaces, we have
\begin{align*}
a^{t}J&=\left<\begin{array}{ccc}
fa^{s+t}b^{r}&:&r\lneqq t, fa^{s}b^{r}\in I\\
fa^{s+t}b^{s+t}&:&fa^{s}b^{s+t}\in I\\
fa^{\ell+t}b^{s+t}&:&fa^{\ell}b^{s+t}\in I \mbox{ or }fa^{s}b^{\ell+t}\in I\\
fa^{s+t}b^{\ell+t}&:&fa^{\ell}b^{s+t}\in I \mbox{ and }fa^{s}b^{\ell+t}\in I
			\end{array}\right>\\
&= \left<\begin{array}{ccc}
fa^{s+t}b^{r}&:&r\lneqq t, fa^{s+t}b^{r}\in a^{t}I\\
fa^{s+t}b^{s+t}&:&fa^{s+t}b^{s+t}\in a^{t}I\\
fa^{\ell+t}b^{s+t}&:&fa^{\ell+t}b^{s+t}\in a^{t}I \mbox{ or
}fa^{s+t}b^{\ell+t}\in a^{t}I\\
fa^{s+t}b^{\ell+t}&:&fa^{\ell+t}b^{s+t}\in a^{t}I \mbox{ and
}fa^{s+t}b^{\ell+t}\in a^{t}I 
			\end{array}\right>\\
&= \left<\begin{array}{ccc}
fa^{s+t}b^{r}&:&r \lneqq t, fa^{s+t}b^{r}\in a^{t}I \mbox{ or }
fa^{r}b^{s+t}\in a^{t}I\\
fa^{s+t}b^{s+t}&:&fa^{s+t}b^{s+t}\in a^{t}I\\
fa^{\ell+t}b^{s+t}&:&fa^{\ell+t}b^{s+t}\in a^{t}I \mbox{ or
}fa^{s+t}b^{\ell+t}\in a^{t}I\\
fa^{s+t}b^{\ell+t}&:&fa^{\ell+t}b^{s+t}\in a^{t}I \mbox{ and
}fa^{s+t}b^{\ell+t}\in a^{t}I 
			\end{array}\right>\\
&=\shiftme_{a,b,0}(a^{t}I).\qedhere
\end{align*}
\end{proof}

We now study the effect of shifting on Betti numbers.
Our main result is the following:

\begin{theorem}\label{strongshiftbetti}
 Let $J=\shiftme_{a,b,t}(I)$.  Then for all $i,j$ one has
  $b_{i,j}(J)\geq b_{i,j}(I)$.  
\end{theorem}

The proof involves several lemmas and sub-propositions.  We begin by
considering the case $t=0$.  Our argument follows
 Murai and Hibi's original proof of Theorem \ref{HM} \cite{HM} very
closely.  In the case that $I$ is squarefree, the arguments are identical.

\begin{definition}
Let $\sigma:S\to S$ be the $\bbbk$-algebra involution defined by
$\sigma(a)=b$, $\sigma(b)=a$, and $\sigma(x_{i})=x_{i}$ for all
$x_{i}\neq a,b$.  
\end{definition}

Since $\sigma$ is an automorphism, it extends to resolutions, and we
have, for example, $b_{i,j}(I)=b_{i,j}(\sigma(I))$ for all graded ideals
$I$.  In fact, $\sigma$ acts naturally on the multigrading, so we have
$b_{i,m}(I)=b_{i,\sigma(m)}(\sigma(I))$ for all monomial ideals.  Note
that $\sigma$ fixes monomials of the form $fa^{s}b^{s}$, and
partitions the other monomials into orbits of cardinality two,
$\sigma(fa^{\ell}b^{s})=fa^{s}b^{\ell}$.

\begin{proposition}\label{symmetry}
 Let $J=\shiftme_{a,b}(I)$.  Then we have $I\cap \sigma(I)=J\cap \sigma(J)$ and
 $I+\sigma(I)=J+\sigma(J)$.  
\end{proposition}
\begin{proof}
Observe that, for any integers $p$ and $q$, we have $fa^{p}b^{q}\in I$
and $fa^{q}b^{p}\in I$ if and only if $fa^{p}b^{q}\in J$ and
$fa^{q}b^{p}\in J$. It follows that $I\cap \sigma(I)=J\cap \sigma(J)$.
Similarly,  $fa^{p}b^{q}\in I$
or $fa^{q}b^{p}\in I$ if and only if $fa^{p}b^{q}\in J$ or
$fa^{q}b^{p}\in J$. It follows that $I+\sigma(I)=J+\sigma(J)$.
%
%
\end{proof}

\begin{lemma}\label{shadowlemma} 
Let $J=\shiftme_{a,b}(I)$, and let $m$ be a monomial
  fixed by $\sigma$.  Then
  $\Shadow_{m}(J)=\shiftme_{a,b}(\Shadow_{m}(I))$.  
\end{lemma}
\begin{proof}
Write $m=fa^{s}b^{s}$, and write $f=g\sqrt{f}$ for some monomial $g$.
For a squarefree monomial $\mu$ dividing $\sqrt{m}$, we have $\mu\in
\Shadow_{m}(J)$ if and only if $\mu\frac{m}{\sqrt{m}}=\mu
ga^{s-1}b^{s-1}\in J$, and similarly 
for $I$.  We write $\mu=f'$, $\mu=f'a$, $\mu=f'b$, or $\mu=f'ab$ with
$f'$ not divisible by $a$ or $b$.  

We consider the case $\mu=f'a$ (the
other three cases are similar).  
%
In this case, we have $\mu=f'a\in\Shadow_{m}(J)$ if and only if   $f'ga^{s}b^{s-1}\in
J$, if and only if $f'ga^{s}b^{s-1}\in I$  or $f'ga^{s-1}b^{s}\in I$,
if and only if 
$f'a\in \Shadow_{m}(I)$  or  $f'b\in
  \Shadow_{m}(I)$, if and only if 
$\mu=f'a\in\shiftme_{a,b}(\Shadow_{m}(I))$. 
%
%
\end{proof}



Theorem \ref{HM} lets us compare the multigraded Betti numbers of $I$
and $J$, in 
multidegrees fixed by $\sigma$:

\begin{corollary}\label{fixedmulti}
Let $J=\shiftme_{a,b}(I)$, and let $m$ be a multidegree fixed by
$\sigma$.  Then for all $i$, one has $b_{i,m}(J)\geq b_{i,m}(I)$.  
\end{corollary}
\begin{proof}
Let $I'=\Shadow_{m}(I)$ and $J'=\Shadow_{m}(J)$.  By Lemma
\ref{shadowlemma}, we have $J'=\shiftme_{a,b}(I')$.  Since $I'$ and $J'$ are
squarefree ideals of $\bbbk[\supp(m)]$, their Betti numbers are
concentrated in squarefree multidegrees.  Thus, in particular,
$b_{i,|\supp(m)|}(I')=b_{i,\sqrt{m}}(I')$ (and likewise for $J'$)
since $\sqrt{m}$ is the only squarefree monomial of degree
$|\supp(m)|$ in this ring.  By Theorem \ref{HM} we have
$b_{i,|\supp(m)|}(J')\geq b_{i,|\supp(m)|}(I')$, and by Theorem
\ref{keylemma} we have $b_{i,m}(I)=b_{i,\sqrt{m}}(I')$ and
$b_{i,m}(J)=b_{i,\sqrt{m}}(J')$.  Putting all this together, we get
$b_{i,m}(J)\geq
b_{i,m}(I)$
 as desired.
\end{proof}

Now we consider multidegrees not fixed by $\sigma$.  
The Mayer-Vietoris sequence,
\begin{equation*}
0\to I\cap \sigma(I)\to I\bigoplus \sigma(I)\to
I+\sigma(I)\to 0,
\end{equation*}
gives rise to a long exact sequence in Tor:
\begin{multline*}
\cdots\to \Tor_{i}\left(\bbbk,{I\cap \sigma(I)}\right) \to
\Tor_{i}\left(\bbbk,{I}\right)\bigoplus
\Tor_{i}\left(\bbbk,{\sigma(I)}\right) \to\\
\Tor_{i}\left(\bbbk,{I+\sigma(I)}\right)\to \allowbreak\Tor_{i-1}\left(\bbbk,
{I\cap \sigma(I)}\right) \to 
\cdots
\end{multline*}

We truncate and restrict to multidegree $m$, producing the exact
sequence of vector spaces:
\begin{multline*}
0\to \left(\ker \Delta_{i,I}\right)_{m}\to
\Tor_{i}\left(\bbbk,{I\cap 
\sigma(I)}\right)_{m}\xrightarrow{\Delta_{i,I}}
\Tor_{i}\left(\bbbk,{I}\right)_{m}\bigoplus 
\Tor_{i}\left(\bbbk,{\sigma(I)}\right)_{m}\\
\to \Tor_{i}\left(\bbbk,{I+\sigma(I)}\right)_{m}\to \left(\ker
\Delta_{i-1,I}\right)_{m}\to 0  
\end{multline*}

\begin{proposition}\label{mayervietoriskernel}
Suppose that $I$ is $(a,b)$-shifted and that $m\neq \sigma(m)$.  Then
$\left(\ker 
\Delta_{i,I}\right)_{m}=0$ for all $i$.
\end{proposition}
\begin{proof}
Suppose $m$ has the form $fa^{s}b^{\ell}$.  (The case
$m=fa^{\ell}b^{s}$ is symmetric.)  Let $g\in \left(\ker\Delta_{i,I}\right)_{m}$ be
given, and write $g=[\sum \alpha_{j}\gamma_{j}e_{\mu_{j}}]$ for some
$\alpha_{j}\in \bbbk$, monomials
$\gamma_{j}\in I\cap \sigma(I)$, and squarefree monomials $\mu_{j}$ of
degree $i$ such that
$\gamma_{j}\mu_{j}=m$ for all $j$.  (The term $\sum
\alpha_{j}\gamma_{j}e_{\mu_{j}}$ is an 
element of ${K}_{i}\tensor (I\cap\sigma(I))$; the brackets denote
its class 
modulo the boundary in the Koszul complex.)  We will show that $g=0$
in $\Tor_{i}(\bbbk,I\cap\sigma(I))$.

We have $\Delta_{i,I}(g)=([g],[g])=(0,0)$ by assumption, so, in
particular, $\sum \alpha_{j}\gamma_{j}e_{\mu_{j}}$ is a boundary in
${K}_{i}\tensor I$.  Thus, we may write $\sum
\alpha_{j}\gamma_{j}e_{\mu_{j}}=D(\sum \beta_{j} h_{j}e_{\nu_{j}})$,
for some coefficients $\beta_{j}\in \bbbk$, monomials $h_{j}\in I$ and, 
$\nu_{j}$ squarefree of degree $i+1$ with $h_{j}\nu_{j}=m$ for all
$j$.

We claim that $h_{j}\in I\cap \sigma(I)$.  Indeed, $h_{j}$ has the
form $f'a^{s-\varepsilon_{a}}b^{\ell-\varepsilon_{b}}$, where $\varepsilon_{a}=0$
if $a$ does not divide $\nu_{j}$ and $1$ if it does, and likewise for
$\varepsilon_{b}$.  Since $\ell\gneqq s$, we have $\ell-\varepsilon_{b}\geq
s-\varepsilon_{a}$, so, since $I$ is shifted,
$f'a^{s-\varepsilon_{a}}b^{\ell-\varepsilon_{b}}\in I \Rightarrow
f'a^{\ell-\varepsilon_{b}}b^{s-\varepsilon_{a}}\in I$.  Thus,
$h_{j}=\sigma(f'a^{\ell-\varepsilon_{b}}b^{s-\varepsilon_{a}})\in \sigma(I)$ as
claimed.  

Hence, $\sum \beta_{j}h_{j}e_{\nu_{j}}\in {K_{i+1}}\tensor (I\cap \sigma(I))$,
so we have $[g]=[D(\sum \beta_{j}h_{j}e_{\nu_{j}})]=0$ in $\Tor_{i}(\bbbk,I\cap\sigma(I))$.  
\end{proof}

\begin{corollary}\label{unfixedmulti}
Let $J=\shiftme_{a,b}(I)$, and let $m$ be a multidegree not fixed by
$\sigma$.  Then for all $i$, one has
$b_{i,m}(J)+b_{i,\sigma(m)}(J)\geq b_{i,m}(I)+b_{i,\sigma(m)}(I)$.  
\end{corollary}
\begin{proof}
From the Mayer-Vietoris sequence, we have 
\begin{align*}
b_{i,m}\left({I}\right)+b_{i,\sigma(m)}\left({I}\right) &=
b_{i,m}\left({I+\sigma(I)}\right) + b_{i,m}\left({I\cap
\sigma(I)}\right)\\
& \quad - \dim_{\bbbk} \left(\ker \Delta_{i,I}\right)_{m} -
\dim_{\bbbk} \left(\ker 
\Delta_{i-1,I}\right)_{m}\\[8pt] 
&= b_{i,m}\left({J+\sigma(J)}\right) + b_{i,m}\left({J\cap
\sigma(J)}\right)\\
& \quad - \dim_{\bbbk} \left(\ker \Delta_{i,I}\right)_{m} -
\dim_{\bbbk} \left(\ker 
\Delta_{i-1,I}\right)_{m}\\[8pt] 
&\leq   b_{i,m}\left({J+\sigma(J)}\right) +
b_{i,m}\left({J\cap \sigma(J)}\right)\\[8pt] 
&= b_{i,m}\left({J+\sigma(J)}\right) + b_{i,m}\left({J\cap
\sigma(J)}\right)\\
&\quad - \dim_{\bbbk} \left(\ker \Delta_{i,J}\right)_{m} -
\dim_{\bbbk} \left(\ker 
\Delta_{i-1,J}\right)_{m}\\[8pt] 
&=b_{i,m}\left({J}\right)+b_{i,\sigma(m)}\left({J}\right),
\end{align*}
the second equality by Proposition \ref{symmetry}, and the fourth by
Proposition \ref{mayervietoriskernel}.
\end{proof}

Corollaries \ref{fixedmulti} and  \ref{unfixedmulti} combine
to prove Theorem \ref{strongshiftbetti} in the case that $t=0$:

\begin{theorem}\label{shiftbetti}
Let $J=\shiftme_{a,b}(I)$.  Then for all $i,j$ one has
  $b_{i,j}(J)\geq b_{i,j}(I)$.  
\end{theorem}
\begin{proof}
We have
\begin{align*}
b_{i,j}(I)&=\hspace{-10pt} \sum_{\deg(m)=j}\hspace{-8pt}b_{i,m}(I)\\
&= \hspace{-10pt}\sum_{\substack{\deg(m)=j\\
    m=fa^{s}b^{s}}}\hspace{-8pt}b_{i,m}(I) 
+\hspace{-10pt}\sum_{\substack{\deg(m)=j\\ 
  m=fa^{\ell}b^{s}}}\hspace{-8pt}\big(b_{i,m}(I)+b_{i,\sigma(m)}(I)\big), 
\end{align*}
and similarly for $J$.  By Corollary \ref{fixedmulti}, the
inequality holds for the first sum, and by Corollary
\ref{unfixedmulti}, it holds for the second.
\end{proof}


The proof of Theorem \ref{strongshiftbetti} is now immediate.

\begin{proof}[Proof of Theorem \ref{strongshiftbetti}]
Let $J=\shiftme_{a,b,t}(I)$.  Then, applying Proposition
\ref{tshiftisshift}, we have
$b_{i,j}(J)=b_{i,j+t}(a^{t}J)=b_{i,j+t}(\shiftme_{a,b}(a^{t}I)) \geq
b_{i,j+t}(a^{t}I)=b_{i,j}(I)$.  
\end{proof}

In fact, this argument, combined with the proof of Theorem
\ref{shiftbetti}, proves the sharper result:

\begin{proposition}\label{sharpbetti}
Let $J=\shiftme_{a,b,t}(I)$.  Then, for all $f$, all $r<t$, and all
$s<\ell$, one has:
\begin{itemize}
\item $b_{i,fa^{s}b^{r}}(J)\geq b_{i,fa^{s}b^{r}}(I)$.
\item $b_{i,fa^{s}b^{s+t}}(J)\geq b_{i,fa^{s}b^{s+t}}(I)$.
\item $b_{i,fa^{s}b^{\ell+t}}(J)+b_{i,fa^{\ell}b^{s+t}}(J) \geq
  b_{i,fa^{s}b^{\ell+t}}(I)+b_{i,fa^{\ell}b^{s+t}}(I)$.  
\end{itemize}
\end{proposition}

\section{Shifted-plus-powers ideals}
The ideal $P=(x_{1}^{e_{1}},\cdots, x_{n}^{e_{n}})$ is $(a,b)$-shifted,
and, furthermore, if $I$ is any 
monomial ideal containing $P$, then $\shiftme_{a,b}(I)$ contains $P$ as
well.  Unfortunately, this statement fails for $(a,b,t)$-shifted ideals.
The goal of this section is to fix this problem.

Let $I$ be a monomial
ideal containing $P$, and write $I=I'+P$.  We will show that, for
appropriate choices of $I'$ (namely, ``deleting'' the pure power of $b$
from a minimal generating set for $I$) and $t$, the
$t$-shifted-plus-$P$ ideal $J=\shiftme_{a,b,t}(I')+P$ has the same
Hilbert function as $I$ and satisfies $b_{i,j}(J)\geq b_{i,j}(I)$.

\begin{notation}
Throughout this section, fix  integers $\beta>1$ and $t\geq 0$.  We
denote by $I$ 
an $(a,b,t)$-shifted ideal with 
no minimal generators divisible by $b^{\beta}$, and set
$J=\shiftme_{a,b,t+1}(I)$.
  By abuse of notation, we will often write
$I+b^{\beta}$ in place of $I+(b^{\beta})$.  
\end{notation}

Our goal is to show that $J+b^{\beta}$
has the same Hilbert function as $I+b^{\beta}$, and larger graded
Betti numbers.

We break down the graded Betti numbers of $I+b^{\beta}$ and
$J+b^{\beta}$ into a sum of multigraded Betti numbers according to the
following formula.
For a monomial $m$ of the form $m=fa^{s}b^{\ell+t+1}$, set
$n=fa^{\ell}b^{s+t+1}$.  Then
\begin{equation}\label{thesummation}
\begin{array}{rcl}
 \displaystyle
    b_{i,j}(J+b^{\beta})&=&\displaystyle\hspace{-15pt}\sum_{\substack{m\neq 
    fa^{s}b^{\ell+t+1}\\m\neq fa^{\ell}b^{s+t+1}}}\hspace{-13pt}
    b_{i,m}(J+b^{\beta}) + 
\hspace{-15pt}\sum_{\substack{m=fa^{s}b^{\ell+t+1}\\ \ell+t+1\neq
  \beta}}\hspace{-13pt}\left(b_{i,m}(J+b^{\beta})+b_{i,n}(J+b^{\beta})\right)\\[12pt]
&&\quad +\displaystyle 
\hspace{-15pt}\sum_{\substack{m=fa^{s}b^{\ell+t+1}\\
    \ell+t+1=\beta}}\hspace{-13pt}
    \left(b_{i,m}(J+b^{\beta})+b_{i,n}(J+b^{\beta})\right),   
\end{array}
\end{equation}
and likewise for $I+b^{\beta}$, all sums taken over monomials $m$
with $\deg m=j$.  We will show that each of the
summands in formula (\ref{thesummation}) for $J$ is larger than or equal
to the
corresponding summand for $I$.

We begin with a technical lemma.

\begin{lemma} \label{techlemma}
  Suppose that $f$ is a monomial not divisible by $a$ or
  $b$, and that $\ell+t+1\geq \beta$.  If $fa^{s}b^{\ell+t+1}\in I$,
  then $fa^{\ell}b^{s+t+1}\in I$ as well.
\end{lemma}
\begin{proof}
  Since $I$ has no minimal generators
  divisible by $b^{\beta}$, we have
  $fa^{s}b^{\ell+t}\in I$.  Since $I$ is
  $(a,b,t)$-shifted, it follows that $fa^{\ell}b^{s+t}\in I$, so
  $fa^{\ell}b^{s+t+1}\in I$ as well.
\end{proof}

\begin{corollary}\label{samecolon}
$I\cap (b^{\beta})=J\cap(b^{\beta})$ and $(I:b^{\beta})=(J:b^{\beta})$.
\end{corollary}
\begin{proof}
Let $m$ be any monomial divisible by $b^{\beta}$, and write $m$ as
$fa^{s}b^{\ell+t+1}$, $fa^{\ell}b^{s+t+1}$, 
$fa^{s}b^{s+t+1}$, or $fa^{s}b^{r}$ with $r\lneqq t+1$, as appropriate.

First, if $m=fa^{s}b^{\ell+t+1}$, then $m\in J$ if and only if $m\in I$ and
$fa^{\ell}b^{s+t+1}\in I$, if and only if (by Lemma \ref{techlemma})
$m\in I$.
Similarly, if $m=fa^{\ell}b^{s+t+1}$, then $m\in J$ if and only if $m\in I$ or
$fa^{s}b^{\ell+t+1}\in I$, if and only if (by Lemma \ref{techlemma})
$m\in I$.
Finally, if $m=fa^{s}b^{s+t+1}$ or $fa^{s}b^{r}$, then $m\in J$ if and only
if $m\in I$.

Thus, $I\cap (b^{\beta})=J\cap (b^{\beta})$, so
$(I:b^{\beta})=(J:b^{\beta})$ as desired.
\end{proof}

This corollary has several important corollaries of its own.

\begin{corollary}\label{samehilb}
$\Hilb(I+b^{\beta})=\Hilb(J+b^{\beta})$.
\end{corollary}

\begin{corollary}
None of the minimal monomial generators of $J$ is divisible by
$b^{\beta+1}$.
\end{corollary}


From the short exact sequence
\begin{equation*}
0\to
\frac{S}{(I:b^{\beta})}(b^{-\beta})\xrightarrow{b^{\beta}}\frac{S}{I}\to
\frac{S}{I+b^{\beta}}\to 0 
\end{equation*}
there arises a long exact sequence in Tor, (the ``mapping cone'',)
\begin{multline*}
0\to \im(b^{\beta}_{*,i,I})\to \tor_{i}\left(\bbbk,\frac{S}{I}\right)\to
\tor_{i}\left(\bbbk,\frac{S}{I+b^{\beta}}\right)\to\\
\tor_{i-1}\left(\bbbk,\frac{S}{(I:b^{\beta})}\right)(b^{-\beta})\to
\im(b^{\beta}_{*,i-1,I})\to 
0,
\end{multline*}
 and similarly for $J$
\begin{multline*}
0\to \im(b^{\beta}_{*,i,J})\to \tor_{i}\left(\bbbk,\frac{S}{J}\right)\to
\tor_{i}\left(\bbbk,\frac{S}{J+b^{\beta}}\right)\to\\
\tor_{i-1}\left(\bbbk,\frac{S}{(J:b^{\beta})}\right)(b^{-\beta})\to
\im(b^{\beta}_{*,i-1,J})\to 0.
\end{multline*}
The following proposition is immediate from mapping cone theory.

\begin{proposition}\label{manyzero}
$\im(b^{\beta}_{*,i,I})=0$ for all $i$, and
$(\im(b^{\beta}_{*,i,J}))_{m}=0$ for all $i$ and all multidegrees $m$
not equal to $fa^{s}b^{\beta}$.  
\end{proposition}
\begin{proof}
Observe that $(I:b^{\beta})$ has no minimal generators divisible by $b$.
Thus, by the Taylor resolution (see e.g.\ \cite{Ei1}*{Exercise 17.11}), its Betti numbers are concentrated in
multidegrees not divisible by $b$, and so
$\tor_{i}(\bbbk,S/(I:b^{\beta}))(b^{-\beta})$ 
is nonzero only in multidegrees of the form $fa^{s}b^{\beta}$.
Furthermore, again by the Taylor resolution, the Betti numbers of
$S/I$ (and so the $\tor_{i}(\bbbk, S/I)$) are concentrated in multidegrees
not divisible by $b^{\beta}$, and those of $S/J$ are concentrated in
multidegrees not divisible by $b^{\beta+1}$.  As the maps
$b^{\beta}_{*,i,I}$ and 
$b^{\beta}_{*,i,J}$ are multihomogeneous, the proposition follows.
%
\end{proof}

\begin{lemma}\label{shadowplusp}
If $m=fa^{s}b^{\beta}$, with $s\geq \beta-t-1$, then
$\Shadow_{m}(J+b^{\beta})=\Shadow_{m}(I+b^{\beta})$.  
\end{lemma}
\begin{proof}
Let $n$ be a monomial dividing $m$, and such that $\frac{m}{n}$ is
squarefree.  We will show that $n\in I+b^{\beta}$ if and only if
$n\in J+b^{\beta}$, 
from which the lemma follows.  
If $b^{\beta}$ divides $n$, then $n\in I+b^{\beta}$ and $n\in
J+b^{\beta}$.  
Otherwise, write $n=f'a^{s}b^{\beta-1}$ (or, mutatis mutandis,
$f'a^{s-1}b^{\beta-1}$).  Then $s\geq (\beta-1)-t-1$, so $n\in J$ if
$n\in I$.  Conversely, if $n\in J$, we have $n\in I$ or
$f'a^{\beta-t-2}b^{s+t+1}\in I$.  In the latter case, $s+t\geq
\beta-1$, so by construction $f'a^{\beta-t-2}b^{\beta-1}\in I$ and so
$f'a^{s}b^{\beta-1}\in I$, i.e., $n\in I$.
%
%
\end{proof}

The following are immediate:
\begin{lemma}\label{smallb}
If $m$ is not divisible by $b^{\beta}$, then
$\shadow_{m}(I+b^{\beta})=\shadow_{m}(I)$ and
$\shadow_{m}(J+b^{\beta})=\shadow_{m}(J)$.  
\end{lemma}

\begin{lemma}\label{bigb}
If $m$ is divisible by $b^{\beta+1}$, then
$\shadow_{m}(I+b^{\beta})=\shadow_{m}(J+b^{\beta})=(1)$.  
\end{lemma}

Using these shadows to compute Betti numbers via Theorem \ref{keylemma}, 
we obtain the following:

\begin{lemma}\label{easy}\ 
\begin{itemize}
\item[(1)] Suppose $m=fa^{s}b^{s+t+1}$ or $fa^{s}b^{r}$ with $r<t+1$.
  Then we have $b_{i,m}(J+b^{\beta})\geq b_{i,m}(I+b^{\beta})$.  
\item[(2)] Suppose $s\lneqq\ell$, and $\ell+t+1\neq \beta$.  Put
  $m=fa^{\ell}b^{s+t+1}$ and $n=fa^{s}b^{\ell+t+1}$.  Then
  $b_{i,m}(J+b^{\beta})+b_{i,n}(J+b^{\beta})\geq
  b_{i,m}(I+b^{\beta})+b_{i,n}(I+b^{\beta})$. 
\end{itemize}
\end{lemma}

\begin{proof}\ 
\begin{itemize}
\item[(1)] If the exponent on $b$ is less than $\beta$, apply Lemma
  \ref{smallb} and Proposition \ref{sharpbetti}.  If it is greater
  than $\beta$, apply Lemma \ref{bigb}.  If the exponent is equal to
  $\beta$, apply Lemma \ref{shadowplusp}.
\item[(2)]  If $s+t+1=\beta$, we have
  $b_{i,n}(I+b^{\beta})=b_{i,n}(J+b^{\beta})$ by 
  Lemma \ref{bigb} and $b_{i,m}(I+b^{\beta})=b_{i,m}(J+b^{\beta})$ by Lemma
  \ref{shadowplusp}.  If $s+t+1\neq \beta$, then, applying the mapping
  cone
  and Proposition \ref{manyzero}, the left-hand side is
  equal to $b_{i,m}(J) + b_{i,n}(J) + b_{i-1,b^{-\beta}m}(J:b^{\beta}) +
  b_{i-1,b^{-\beta}n}(J:b^{\beta})$, while the right-hand side is equal to
  $b_{i,m}(I) + b_{i,n}(I) + b_{i-1,b^{-\beta}m}(I:b^{\beta}) + 
  b_{i-1,b^{-\beta}n}(I:b^{\beta})$.  Apply Proposition \ref{sharpbetti} and
  Corollary \ref{samecolon}.\qedhere
\end{itemize}
\end{proof}

Thus, the first two sums in formula (\ref{thesummation}) are larger
for $J+b^{\beta}$ than for $I+b^{\beta}$.
It remains to consider the case that $m=fa^{s}b^{\ell+t+1}$, with
$\ell+t+1=\beta$.  
We fix $m=fa^{s}b^{\beta}$ with $\beta=\ell+t+1\gneqq s+t+1$, 
multiply $I$ by
$a^{t+1}$, and recall the Mayer-Vietoris sequence 
from the previous section:
\begin{multline*}
0\to (\ker \Delta_{i,a^{t+1}I})_{a^{t+1}{m}}\to
\tor_{i}\left(\bbbk,{a^{t+1}I\cap\sigma(a^{t+1}I)}\right)_{a^{t+1}{m}}  
\xrightarrow{\Delta_{i,a^{t+1}I}}\\
\tor_{i}\left(\bbbk,{a^{t+1}I}\right)_{a^{t+1}{m}}
\bigoplus
\tor_{i}\left(\bbbk,{\sigma(a^{t+1}I)}\right)_{a^{t+1}{m}}\to\\ 
\tor_{i}\left(\bbbk,{a^{t+1}I+\sigma(a^{t+1}I)}\right)_{a^{t+1}{m}}\to  
(\ker\Delta_{i-1, a^{t+1}I})_{a^{t+1}{m}}\to 0.
\end{multline*}

\begin{lemma}\label{sameshadow}
$\shadow_{a^{t+1}m}(a^{t+1}J) =
  \shadow_{a^{t+1}m}(a^{t+1}I\cap \sigma(a^{t+1}I))$.   
\end{lemma}
\begin{proof} 
Let $n$ be a monomial dividing $a^{t+1}m$, and such that
  $\frac{a^{t+1}m}{n}$ is 
  squarefree.   We will show that $n\in a^{t+1}J$ if and only if $n\in
  a^{t+1}I\cap \sigma(a^{t+1}I)$, from which the lemma follows.  
We may write
$n=f'a^{s+t+1-\varepsilon_{a}}b^{\ell+t+1-\varepsilon_{b}}$ with
  $\varepsilon_{a}, \varepsilon_{b}=0$ or $1$ (so
  $s-\varepsilon_{a}\leq \ell-\varepsilon_{b}$).  
By definition $n\in a^{t+1}J$ if and only if $n\in a^{t+1}I$ and
$f'a^{\ell+t+1-\varepsilon_{b}}b^{s+t+1-\varepsilon_{a}}\in a^{t+1}I$,
if and only if $n\in a^{t+1}I$ and $n\in \sigma(a^{t+1}I)$.
\end{proof}

\begin{corollary}\label{theisomorphism}
$\displaystyle\Tor_{i}\left(\bbbk,{J}\right)_{m}\cong
  \tor_{i}\left(\bbbk, {a^{t+1}I\cap
  \sigma(a^{t+1}I)}\right)(a^{t+1})_{m}$.  
\end{corollary}

\begin{proof} From the proof of Theorem \ref{keylemma}, the complex
  $(\mathbb{K}_{\bullet}\tensor M)_{m}$ depends only on
  $\shadow_{m}(M)$ for any monomial ideal $M$ and multidegree $m$.

  We have $(\mathbb{K}_{\bullet}\tensor J)_{m}\cong
  (\mathbb{K}_{\bullet}\tensor a^{t+1}J)(a^{t+1})_{m} =
  (\mathbb{K}_{\bullet}\tensor
  (a^{t+1}I\cap\sigma(a^{t+1}I)))(a^{t+1})_{m}$, the first isomorphism
  given by multiplication by $a^{t+1}$, the second equality by
  applying Lemma \ref{sameshadow}.  This isomorphism of
  complices induces an isomorphism on $\tor$,
  \[\phi_{i,m}:\tor_{i}(\bbbk,J)_{m}\to \tor_{i}(\bbbk,
  a^{t+1}I\cap\sigma(a^{t+1}I))(a^{t+1})_{m}\] given by
  $\phi_{i,m}([g])=[a^{t+1}g]$ for any cycle $[g]\in K_{i}\tensor
  J$.
\end{proof}

We view $\im(b^{\beta}_{*,i,J})_{m}$ 
  and $(\ker \Delta_{i-1,a^{t+1}I})(a^{t+1})_{m}$ as submodules of
  $\tor_{i-1}(\bbbk,J)$ (via the natural
  isomorphism with $\tor_{i}(\bbbk,S/J)$)
  and of
  $\tor_{i-1}(\bbbk,a^{t+1}I\cap 
  \sigma(a^{t+1}I))(a^{t+1})_{m}$, respectively.
  The isomorphism $\phi_{i,m}$ allows us to compare these two vector
  spaces.  

\begin{proposition}\label{smaller}
$\phi_{i,m}(\im(b^{\beta}_{*,i,J})_{m})\subset (\ker
  \Delta_{i-1,a^{t+1}I})(a^{t+1})_{m}$.
\end{proposition}

\begin{proof}
  An element of $\im(b^{\beta}_{*,i,J})_{m}$ has the form
  $[b^{\beta}g]$, where $g$ is a cycle in ${K}_{i-1}\tensor
  (J:b^{\beta})$.  (Consider e.g.\ the connecting homomorphism arising
  from the short exact sequence $0\to J\to S\to S/J\to 0$.)
We have $\phi_{i,m}([b^{\beta}g])=[a^{t+1}b^{\beta}g]$.  

  To show that $[a^{t+1}b^{\beta}g]\in
  (\ker\Delta_{i-1,a^{t+1}I})(a^{t+1})_{m}$, it 
  suffices to show that $a^{t+1}b^{\beta}g$ is a boundary in both
  $a^{t+1}I\tensor {K}_{i-1}$ and $\sigma(a^{t+1}I)\tensor
  {K}_{i-1}$.  From the Taylor resolution of $I$, we know that
  $a^{t+1}I\tensor\mathbb{K}_{\bullet}$ is
  exact in multidegree $a^{t+1}m$.  Thus, since  $a^{t+1}b^{\beta}g$
  is a cycle in 
  $a^{t+1}I\tensor {K}_{i-1}$, it is a boundary as well.  Hence, we may
  write $a^{t+1}b^{\beta}g=D(h)$ for some $h\in
  a^{t+1}I\tensor{K}_{i}$.   

  Write $h=a^{s+t}b^{\beta-1}e_{a}\wedge e_{b}\wedge f_{1} +
  a^{s+t+1}b^{\beta-1}e_{b}\wedge f_{2} + a^{s+t}b^{\beta}e_{a}\wedge
  f_{3} + a^{s+t+1}b^{\beta}f_{4}$, for $f_{1},f_{2},f_{3},f_{4}\in
  \mathbb{K}_{\bullet}$ not involving $a$, $b$, $e_{a}$, or $e_{b}$.  
  Then, write
  $a^{s+t+1}b^{\beta} f_{4}$ (and, mutatis mutandis,
  $a^{s+t}b^{\beta}e_{a}\wedge f_{3}$) in the form  $\sum \alpha_{j}
  a^{s+t+1}b^{\beta}\gamma_{j}e_{\mu_{j}}$ for coefficients
  $\alpha_{j}\in \bbbk$ and  monomials
  $\gamma_{j}$ with $a^{s+t+1}b^{\beta}\gamma_{j}\in a^{t+1}I$, and
  hence 
  $a^{s}b^{\beta}\gamma_{j}\in I\cap (b^{\beta})=J\cap(b^{\beta})$.
  Adjusting $b^{\beta}g$ in 
  $\im(b^{\beta})_{*,i,J}$ if necessary, we may assume that
  $f_{3}=f_{4}=0$.  

  Thus, \begin{align*}
a^{t+1}b^{\beta}g&=D(h)\\
&= a^{s+t+1}b^{\beta-1}e_{b}\wedge f_{1}
  - a^{s+t}b^{\beta}e_{a}\wedge f_{1}
 + a^{s+t}b^{\beta-1}e_{a}\wedge
  e_{b}\wedge D(f_{1})\\
&\quad  + a^{s+t+1}b^{\beta}f_{2} -
  a^{s+t+1}b^{\beta-1}e_{b}\wedge D(f_{2}).
\end{align*}
  Since the left-hand side of this expression
  is divisible by $b^{\beta}$, it follows that both
  $a^{s+t+1}b^{\beta-1}e_{b}\wedge f_{1}  -
  a^{s+t+1}b^{\beta-1}e_{b}\wedge D(f_{2})$  and
  $a^{s+t}b^{\beta-1}e_{a}\wedge 
  e_{b}\wedge D(f_{1})$ are equal to zero, and, in particular,
  $f_{1}=D(f_{2})$ (and 
  $D(f_{1})=0$).  
  Thus,
\begin{align*}
  a^{t+1}b^{\beta}g&=a^{s+t+1}b^{\beta}f_{2}-a^{s+t}b^{\beta}e_{a}\wedge
  D(f_{2}) \\
&= D(a^{s+t}b^{\beta}e_{a}\wedge f_{2}).
\end{align*}

  We claim that this is a boundary in $\sigma(a^{t+1}I)\tensor
  {K}_{i-1}$.  Indeed, we may write $f_{2}$  in the form $\sum
  \alpha_{j}\gamma_{j}e_{\mu_{j}}$ with $a^{s+t+1}b^{\beta-1}\gamma_{j}\in
  a^{t+1}I$, i.e.,
  $a^{s}b^{\beta-1}\gamma_{j}=a^{s}b^{\ell+t}\gamma_{j}\in I$.  Then,
  since $I$ is $(a,b,t)$-shifted, we have
  $a^{\ell}b^{s+t}\gamma_{j}\in I$, so
  $a^{\ell+t+1}b^{s+t}\gamma_{j}\in a^{t+1}I$ and
  $a^{s+t}b^{\ell+t+1}\gamma_{j}=a^{s+t}b^{\beta}\gamma_{j}\in
  \sigma(a^{t+1}I)$.  Thus, $a^{s+t}b^{\beta}e_{a}\wedge f_{2}\in
  \sigma(a^{t+1}I)\tensor {K}_{i}$ as desired.
\end{proof}



\begin{corollary}\label{hard}
 For $m=fa^{s}b^{\beta}=fa^{s}b^{\ell+t+1}$, set
  $n=fa^{\ell}b^{s+t+1}$.  Then, for all $i$, one has
  $b_{i,m}(S/(J+b^{\beta}))+b_{i,n}(S/(J+b^{\beta}))\geq
  b_{i,m}(S/(I+b^{\beta}))+b_{i,n}(S/(I+b^{\beta}))$.   
\end{corollary}
\begin{proof}
The computation below appears daunting, but it is in fact merely
long.  The moral is that, by Proposition \ref{smaller}, the
flexibility  in the paired multidegrees $m$ and
$n$ (given by $(\ker \Delta_{\bullet,
  a^{t+1}I})_{a^{t+1}m}$) is larger than the obstruction coming from the
cancellation in the mapping cone (given by $\im
b^{\beta}_{*,\bullet,J}$).  

Set $A=b_{i,m}(S/(J+b^{\beta}))+b_{i,n}(S/(J+b_{\beta})) -
b_{i,m}(S/(I+b^{\beta})) - b_{i,n}(S/(I+b^{\beta}))$.  We will show that
$A$ is nonnegative.  

Expanding each term of $A$ with the mapping cone, we have

\begin{align*}
A&=\left(b_{i,m}\left(\frac{S}{J}\right) +
b_{i-1,b^{-\beta}m}\left(\frac{S}{(J:b^{\beta})}\right) -
\dim_{\bbbk}\left(\im b^{\beta}_{*,J,i}\right)_{m} - \dim_{\bbbk}\left(\im
b^{\beta}_{*,J,i-1}\right)_{m}\right)\\
&\quad + \left(b_{i,n}\left(\frac{S}{J}\right) +
b_{i-1,b^{-\beta}n}\left(\frac{S}{(J:b^{\beta})}\right) -
\dim_{\bbbk}\left(\im b^{\beta}_{*,J,i}\right)_{n} - \dim_{\bbbk}\left(\im
b^{\beta}_{*,J,i-1}\right)_{n}\right)\\
&\quad - \left(b_{i,m}\left(\frac{S}{I}\right) +
b_{i-1,b^{-\beta}m}\left(\frac{S}{(I:b^{\beta})}\right) -
\dim_{\bbbk}\left(\im b^{\beta}_{*,I,i}\right)_{m} - \dim_{\bbbk}\left(\im
b^{\beta}_{*,I,i-1}\right)_{m}\right)\\
&\quad - \left(b_{i,n}\left(\frac{S}{I}\right) +
b_{i-1,b^{-\beta}n}\left(\frac{S}{(I:b^{\beta})}\right) -
\dim_{\bbbk}\left(\im b^{\beta}_{*,I,i}\right)_{n} - \dim_{\bbbk}\left(\im
b^{\beta}_{*,I,i-1}\right)_{n}\right).
\end{align*}
By Proposition \ref{manyzero}, most of these images are empty, and by
Corollary \ref{samecolon}, the Betti numbers of the colon ideals all
cancel.  We are left with

\begin{align*}
A&=\left(b_{i,m}\left(\frac{S}{J}\right) +
b_{i,n}\left(\frac{S}{J}\right)\right) -
\left(b_{i,m}\left(\frac{S}{I}\right) +
b_{i,n}\left(\frac{S}{I}\right)\right)\\
&\quad - \dim_{\bbbk}\left(\im b^{\beta}_{*,J,i}\right)_{m} - \dim_{\bbbk}\left(\im
b^{\beta}_{*,J,i-1}\right)_{m}.
\end{align*}
We multiply the ideals by $a^{t+1}$ (replacing
$b_{i,m}\left(\frac{S}{J}\right)$ with
$b_{i,a^{t+1}m}\left(\frac{S}{a^{t+1}J}\right)$, etc.), and then
expand again with the Mayer-Vietoris sequence, yielding

\begin{align*}
A&=\bigg[b_{i,a^{t+1}m}\left(\frac{S}{a^{t+1}J\cap\sigma(a^{t+1}J)}\right)
+ b_{i,a^{t+1}m}\left(\frac{S}{a^{t+1}J + \sigma(a^{t+1}J)}\right)\\
&\qquad\qquad\qquad  - \dim_{\bbbk}\left(\ker
\Delta_{i-1,a^{t+1}J}\right)_{a^{t+1}m} -\dim_{\bbbk}\left(\ker 
\Delta_{i-2,a^{t+1}J}\right)_{a^{t+1}m}\bigg]\\
&\quad -
\bigg[b_{i,a^{t+1}m}\left(\frac{S}{a^{t+1}I\cap\sigma(a^{t+1}I)}\right) 
+ b_{i,a^{t+1}m}\left(\frac{S}{a^{t+1}I + \sigma(a^{t+1}I)}\right)\\
&\qquad\qquad\qquad  - \dim_{\bbbk}\left(\ker
\Delta_{i-1,a^{t+1}I}\right)_{a^{t+1}m} -\dim_{\bbbk}\left(\ker 
\Delta_{i-2,a^{t+1}I}\right)_{a^{t+1}m}\bigg]\\
&\quad - \dim_{\bbbk}\left(\im b^{\beta}_{*,J,i}\right)_{m} - \dim_{\bbbk}\left(\im
b^{\beta}_{*,J,i-1}\right)_{m}.
\end{align*}
The remaining Betti numbers cancel by Propositions \ref{tshiftisshift} and
\ref{symmetry}, and 
the first two kernels are empty by Proposition
\ref{mayervietoriskernel}.  We are left with

\begin{align*}
A&=\left(\dim_{\bbbk}\left(\ker
\Delta_{i-1,a^{t+1}I}\right)_{a^{t+1}m} - \dim_{\bbbk}\left(\im 
b^{\beta}_{*,J,i}\right)_{m}\right)\\ 
&\quad +\left(\dim_{\bbbk}\left(\ker
\Delta_{i-2,a^{t+1}I}\right)_{a^{t+1}m} - \dim_{\bbbk}\left(\im 
b^{\beta}_{*,J,i-1}\right)_{m}\right).
\end{align*}
By Proposition \ref{smaller}, each of these summands is nonnegative.
\end{proof}

\begin{proposition}\label{shiftpluspbetti}
For all $i,j$, one has $b_{i,j}(J+b^{\beta})\geq
b_{i,j}(I+b^{\beta})$.  
\end{proposition}
\begin{proof}
For a monomial $m$ of the form $m=fa^{s}b^{\ell+t+1}$, set
$n=fa^{\ell}b^{s+t+1}$.  We recall formula (\ref{thesummation}),
\begin{align*}
b_{i,j}(J+b^{\beta})&=\hspace{-15pt}\sum_{\substack{m\neq
    fa^{s}b^{\ell+t+1}\\m\neq fa^{\ell}b^{s+t+1}}}\hspace{-13pt}
    b_{i,m}(J+b^{\beta}) + 
\hspace{-15pt}\sum_{\substack{m=fa^{s}b^{\ell+t+1}\\ \ell+t+1\neq
  \beta}}\hspace{-13pt}\left(b_{i,m}(J+b^{\beta})+b_{i,n}(J+b^{\beta})\right)\\
&\quad + 
\hspace{-15pt}\sum_{\substack{m=fa^{s}b^{\ell+t+1}\\
    \ell+t+1=\beta}}\hspace{-13pt}
    \left(b_{i,m}(J+b^{\beta})+b_{i,n}(J+b^{\beta})\right),   
\end{align*}
and similarly for $I$.  By Lemma \ref{easy}, the inequality holds for
the first two sums, and by Corollary \ref{hard}, it holds for the third.
\end{proof}

\section{Strongly shifted ideals and compression}

Let $J$ be an $(a,b)$-strongly shifted ideal, none of whose generators
is divisible by $b^{\beta}$, and suppose further that $J$ contains
$a^{\alpha}$ for some $\alpha\leq \beta$.  Let $T$ be the
$\{a,b\}$-compression of $J$.  We study the Betti numbers of $J$ and
$T$.  We continue to denote by $f$  a monomial not
divisible by $a$ or $b$.

The following observations are immediate:

\begin{lemma}\label{compressionstuff}\  
\begin{itemize}
\item[(i)] If $T\neq J$, then $T$ is reverse lexicographically greater
  than $J$.
\item[(ii)] $fa^{r}b^{s}\in T$ if and only if $J$ contains at least $s+1$
  monomials of the form $fa^{p}b^{q}$ with $p+q=r+s$.
\end{itemize}
\end{lemma}

\begin{lemma}\label{compressionlemma}
The following are equivalent:
\begin{itemize}
\item[(i)] $fa^{r}b^{\beta}\in J$.
\item[(ii)] $fa^{r}b^{\beta-1}\in J$.
\item[(iii)] $fa^{p}b^{q}\in J$ for all $p,q$ such that
  $p+q=r+\beta-1$ and $q\lneqq\beta$.  
\item[(iv)] $fa^{r}b^{\beta-1}\in T$.
\item[(v)] $fa^{r}b^{\beta}\in T$.
\end{itemize}
\end{lemma}
\begin{proof}
(iii)$\implies$(ii)$\implies$(i) is obvious, as is
  (iii)$\implies$(iv)$\implies$(v), and (i)$\implies$(ii) is 
  immediate by construction.  We will show that (ii) implies (iii) and
  (v) implies (i).

Suppose (ii) holds.  If $p\geq\alpha$, then $fa^{p}b^{q}\in J$ because
$a^{\alpha}\in J$. 
Otherwise, note that $p\geq r$, since $p-r=\beta-1-q\geq 0$ by
assumption.  Set $t=\beta-1-p$, which is nonnegative since
$\beta-1\geq \alpha-1\geq p$.  Since $J$
is $(a,b,t)$-shifted and $fa^{r}b^{\beta-1}=fa^{r}b^{t+p}\in
J$, we have 
$fa^{p}b^{t+r}=fa^{p}b^{q}\in J$.  Thus (iii) holds.

Now suppose (v) holds.  Then, by Lemma \ref{compressionstuff} (ii),
$J$ contains 
at least $\beta+1$ monomials of the form $fa^{p}b^{q}$, with
$p+q=\beta+r$.  By the pigeonhole principle, one of these, say
$fa^{P}b^{Q}$, has $Q\geq \beta$ and $P\leq r$.  By construction,
then, $J$ contains $fa^{P}b^{\beta}$ and so we have
$fa^{r}b^{\beta}\in J$, and (i) is satisfied. 
\end{proof}

\begin{corollary}\label{nocancellation}
No minimal monomial generator of $T$ is divisible by $b^{\beta}$.
\end{corollary}
\begin{proof}
Suppose that $T$ contains a monomial $m$ of the form
$fa^{r}b^{d}$ with $d\geq \beta$.  
No minimal monomial generator of $J$ is divisible by $b^{d}$, so Lemma
\ref{compressionlemma} applies with $d$ in place of $\beta$, and we
have $fa^{r}b^{d-1}\in T$.  Thus, $m$ is not a minimal
generator of $T$.
\end{proof}

\begin{corollary}\label{compressionsamecolon}
We have $T\cap (b^{\beta})=J\cap(b^{\beta})$ and
$(T:b^{\beta})=(J:b^{\beta})$.  
\end{corollary}
\begin{proof}
For any $q\geq\beta$, one has $fa^{p}b^{q}\in J$ if and only if
$fa^{p}b^{\beta-1}\in J$, if and only if $fa^{p}b^{\beta-1}\in T$, if
and only if $fa^{p}b^{q}\in T$.  
\end{proof}

\begin{proposition}\label{compressedpluspbetti}
$b_{i,j}(T+b^{\beta})\geq b_{i,j}(J+b^{\beta})$ for all $i,j$.
\end{proposition}
\begin{proof}
Both $T+b^{\beta}$ and $J+b^{\beta}$ are resolved by the mapping
cone of $b^{\beta}$, via the short exact sequences
\[
0\to S/(J:b^{\beta})(b^{-\beta})\to S/J\to S/(J+b^{\beta})\to 0
\]
and
\[
0\to S/(T:b^{\beta})(b^{-\beta})\to S/T\to S/(T+b^{\beta})\to 0.
\]

By construction (for $J$) and Corollary \ref{nocancellation} (for
$T$), neither $J$ nor $T$ has any minimal generators divisible by
$b^{\beta}$, so, by the Taylor resolution, their multigraded Betti
numbers are concentrated in multidegrees not divisible by
$b^{\beta}$.  Thus, there is no cancellation in either mapping cone,
and we have
\begin{align*}
b_{i,j}(S/(J+b^{\beta}))&=b_{i,j}(S/J)+b_{i-1,j-\beta}(S/(J:b^{\beta}))\\
b_{i,j}(S/(T+b^{\beta}))&=b_{i,j}(S/T)+b_{i-1,j-\beta}(S/(T:b^{\beta})).
\end{align*}
By Theorem \ref{compressionbetti}, we have $b_{i,j}(S/T)\geq
b_{i,j}(S/J)$, and by Corollary \ref{compressionsamecolon},
$(J:b^{\beta})=(T:b^{\beta})$.  
\end{proof}

\section{The monomial case of the lex-plus-powers conjecture}
In this section, we put everything together to prove the monomial case of the
lex-plus-powers conjecture in arbitrary characteristic: 

\begin{theorem}\label{lpp}
Suppose $I$ is a homogeneous ideal containing a regular sequence of
monomials 
$(f_{1},\cdots,f_{r})$ in degrees $e_{1}\leq\cdots\leq
e_{r}$.  Put 
$P=(x_{1}^{e_{1}},\cdots, x_{r}^{e_{r}})$.  Then there exists a lex-plus-$P$
ideal $L$ such that $L$ has the same Hilbert function as $I$ and
$b_{i,j}(L)\geq b_{i,j}(I)$ for all $i,j$.
\end{theorem}

 Throughout the section, $e_{1}\leq
\cdots \leq e_{n}\leq\infty$.  For variables $a=x_{i},
b=x_{j}$, set $e_{a}=e_{i}$ and $e_{b}=e_{j}$.  For ideals containing
$P=(x_{1}^{e_{1}},\cdots,x_{n}^{e_{n}})$, we will frequently want to
consider the ideal obtained by ``deleting'' $b^{e_{b}}$ from a
generating set:

\begin{notation}
Let a monomial ideal $I\supset P$ and variables $a,b$ be given.  We
set 
$I'$ equal to the ideal 
generated by all the minimal monomial generators of $I$ except for
$b^{e_{b}}$.  (If $b^{e_{b}}$ is not a minimal monomial generator of
$I$, set $I'=I$.)
\end{notation}

\begin{lemma} Let $a$ and $b$ be given.  Then $I$ is
  $(a,b,t)$-shifted-plus-$P$ if and only if $I'$ is 
  $(a,b,t)$-shifted, and $I$ is $\{a,b\}$-compressed-plus-$P$ if and
  only if $I'$ is $\{a,b\}$-compressed.
\end{lemma}
\begin{proof}  Suppose that $I$ is $\{a,b\}$-compressed-plus-$P$.
  (The proof for $(a,b,t)$-shifted is similar.)  Then there exists an
  $\{a,b\}$-compressed 
  ideal $\hat{I}$ such that $I=\hat{I}+P$.  We will show that $I'$ is
  $\{a,b\}$-compressed.  

  Fix a monomial $f$ not divisible by $a$ or $b$, and suppose that
  $u=fa^{p}b^{q}$ and $v=fa^{r}b^{s}$ are monomials of the same degree
  such that $v>_{\mathrm{lex}} u$ (i.e., $r>p$) and $u\in I'$.  We
  need to show that $v\in I'$ as well.  If $p\geq e_{a}$, we have
  $v\in I'$ since $a^{e_{a}}\in I'$.  Otherwise, $u$ is divisible by
  some minimal generator $w$ of $I'$; write $w=f'a^{p'}b^{q'}$ with
  $q'<e_{b}$.  If
  $q'<s$ or $w=x_{k}^{e_{k}}$, we have $v\in I'$ immediately,
  otherwise $w\in \hat{I}$ 
  since $w\in I\smallsetminus P$.  Since $\hat{I}$ is
  $\{a,b\}$-compressed, it follows that $f'a^{p'+q'-s}b^{s}\in
  \hat{I}$.  Since this is in $I$ but is not divisible by $b^{e_{b}}$,
  it is in $I'$, so we have $v\in I'$ as desired.
\end{proof}

\begin{proposition}\label{tshiftingplusp}
Let $I$ be a monomial ideal which is $(a,b,t)$-shifted-plus-$P$.
Then there exists an $(a,b,t+1)$-shifted-plus-$P$ ideal $J$ which 
has the same
Hilbert function as 
$I$, is reverse lexicographically greater than or equal to $I$, and satisfies
$b_{i,j}(J)\geq b_{i,j}(I)$. 
\end{proposition}
\begin{proof}
Set $J'=\shiftme_{a,b,t+1}(I')$ and $J=J'+P$.  
We have $I=I'+b^{e_{b}}$ and $J=J'+b^{e_{b}}$, so, by Corollary
\ref{samehilb}, $I$ and $J$ have the same Hilbert function, by
Proposition \ref{shiftmacaulay}, $J$ is reverse lexicographically
greater than or equal to $I$, and, by
Proposition \ref{shiftpluspbetti}, $b_{i,j}(J)\geq b_{i,j}(I)$.
\end{proof}

\begin{proposition}\label{stronglyshifting}
Let $I$ be a monomial ideal containing $P$, and fix $a$ and $b$.  Then
there exists an $(a,b)$-strongly shifted-plus-$P$ ideal $J$ which is
reverse lexicographically greater than or equal to $I$, has the same
Hilbert function as $I$, and satisfies $b_{i,j}(J)\geq b_{i,j}(I)$ for
all $i,j$.  
\end{proposition}
\begin{proof}
Clearly, $\shiftme_{a,b}(I)$ contains $P$.  Thus, replacing $I$ with
$\shiftme_{a,b}(I)$ if necessary (and applying Proposition
\ref{shiftmacaulay} and Theorem \ref{shiftbetti},) we may assume that $I$ is 
$(a,b)$-shifted.  
If $I$ is not already $(a,b)$-strongly shifted-plus-$P$, there exist integers
$t\gneqq 0$ such
that $I$ is not $(a,b,t)$-shifted-plus-$P$.  Choose the smallest such
$t$.   Then
by Proposition 
\ref{tshiftingplusp} there exists an $(a,b,t)$-shifted-plus-$P$
ideal with the same Hilbert function as $I$ and larger graded Betti
numbers.  Replace $I$ with this new ideal and repeat.  This process
must terminate, since there are only finitely many monomial ideals
with the same Hilbert function, and at each step we replace the ideal
with a reverse lexicographically greater one.
Let $J$ be the
resulting ideal.
\end{proof}

\begin{proposition}\label{compressionplusp}
Let $I$ be $(a,b)$-strongly-shifted-plus-$P$.  
Then there exists
an
$\{a,b\}$-compressed-plus-$P$ ideal $T$ which is reverse lexicographically
greater than or equal to $I$, has the same Hilbert
function as 
$I$, and satisfies $b_{i,j}(T)\geq b_{i,j}(I)$.  
\end{proposition}
\begin{proof}
Let $T'$ be the $\{a,b\}$-compression of $I'$, and put $T=T'+P$.  
We have $I=I'+b^{e_{b}}$ and $T=T'+b^{e_{b}}$, so, by Corollary
\ref{compressionsamecolon}, $I$ and $T$ have the same Hilbert function,
and, by Proposition \ref{compressedpluspbetti}, $b_{i,j}(T)\geq
b_{i,j}(I)$.
\end{proof}

\begin{proposition}\label{borellification}
Let $I$ be a monomial ideal containing $P$.  Then there exists a
Borel-plus-$P$ ideal $B$ such that $B$ has the same Hilbert function
as $I$, and $b_{i,j}(B)\geq b_{i,j}(I)$ for all $i,j$.
\end{proposition}
\begin{proof}
If $I$ is not already Borel-plus-$P$, there exist pairs of variables
$a,b$ such that $I$ is not $\{a,b\}$-compressed-plus-$P$.  Choose any such
pair.  By Propositions
\ref{stronglyshifting} 
and \ref{compressionplusp}, there exists an
$\{a,b\}$-compressed-plus-$P$ ideal $T$ with the same Hilbert function as $I$
and larger Betti numbers.  Replace $I$ with $T$ and repeat.  This
process must terminate because there are only finitely many monomial
ideals with the same Hilbert function, and at each step we are
replacing the ideal with a reverse lexicographically greater one.
Let $B$ be the resulting ideal.
\end{proof}

Theorem \ref{boreltolex} completes the proof of Theorem \ref{lpp}.

\begin{proof}[Proof of Theorem \ref{lpp}.]
By Lemma \ref{anymonregseq}, we may assume without loss of
generality that
$(f_{1},\cdots,f_{r})=P$.  By Proposition \ref{borellification}, we
may assume that $I$ is Borel-plus-$P$.  Thus, the desired inequality holds by
Theorem \ref{boreltolex}.
\end{proof}

\section{Consecutive Cancellation}

A \emph{consecutive cancellation} in the graded Betti numbers of a module $M$
is the simultaneous subtraction of $1$ from consecutive Betti numbers
in the same internal degree, i.e., replacing $b_{i,j}(M)$ and
$b_{i-1,j}(M)$ with $(b_{i,j}(M)-1)$ and $(b_{i-1,j}(M)-1)$.  

We say that the graded Betti numbers of an ideal $I$ are obtained from those
of $L$ by consecutive cancellations if we can perform a sequence of
consecutive cancellations on the $b_{i,j}(L)$ to produce the Betti
numbers of $I$.  Heuristically, this happens because the minimal
resolution of $L$ ``deforms'' into a (non-minimal) resolution of $I$,
which can be decomposed into a direct sum of the minimal resolution of
$I$ and some trivial complices $0\to S\to S\to 0$; the cancellations
are in the degrees of these trivial complices.
We define this more formally as follows:  

\begin{definition}
Let $L$ and $I$ be two homogeneous ideals.  We say that the graded Betti
numbers of $I$ are \emph{obtained from those of $L$ by consecutive
  cancellations} if there exist nonnegative integers $c_{i,j}$ such
that, for all 
$i$ and $j$, we have $b_{i,j}(I)=b_{i,j}(L)-c_{i,j}-c_{i-1,j}$.  
\end{definition}

Peeva shows in \cite{Pe} that, if $L$ is the lex ideal with the same
Hilbert function as $I$, the graded Betti numbers of $I$ are obtained from
those of $L$ by consecutive cancellations; similar results are known
(often with the same proof) in many settings where the lex
ideals attain all Hilbert functions.

We will show:
\begin{theorem}\label{consecutivecancellations}
Let $I$, $P$, and $L$ be as in the statement of
  Theorem \ref{lpp}.  Then the graded Betti numbers of $I$ are obtained from
  those of $L$ by consecutive cancellations.
\end{theorem}
\begin{proof}
The proof of Theorem \ref{lpp} consists of a series of compressions,
shifts, and $t$-shifts-plus-$P$, followed by a jump from
Borel-plus-$P$ to lex-plus-$P$.  We will show that at each step the
graded Betti numbers are obtained by consecutive cancellations.  Thus, what
we must show is that the graded Betti numbers of $I$ are obtained from
those of $J$ (or $T$) in Theorem \ref{boreltolex}, Lemma
\ref{anymonregseq}, and Propositions
\ref{stronglyshifting} and \ref{compressionplusp}.

Murai shows in \cite{Mu1}*{Theorem 5.1} that the Betti numbers of a
Borel-plus-$P$ 
ideal are obtained from those of the lex-plus-$P$ ideal by consecutive
cancellations.  Since Lemma \ref{anymonregseq} and Proposition
\ref{compressionplusp} use 
only coordinate changes, initial ideals, and compressions, the
statement follows from \cite{Pe} and \cite{Me3}*{Theorem 5.10} in
these cases.  Lemma \ref{theend} below completes the proof.
\end{proof}

\begin{lemma}\label{theend}
Let $I$ be a monomial ideal.
\begin{itemize}
\item[(1)]  Set $J=\shiftme_{a,b}(I)$.  Then the graded Betti numbers
  of $I$ are obtained from those of $J$ by consecutive cancellations.
\item[(2)]  Suppose that $I$ is $(a,b,t)$-shifted and has no
  generators divisible by $b^{\beta}$.  Set $J=\shiftme_{a,b,t}(I)$.
  Then the graded Betti numbers of $I+b^{\beta}$ are obtained from
  those of $J+b^{\beta}$ by consecutive cancellations.
\end{itemize}
\end{lemma}
\begin{proof}\ 
\begin{itemize}
\item[(1)]
%
Let $m$ be a multidegree.  
If $m$ has the form $fa^{s}b^{s}$, we have
$b_{i,m}(I)=b_{i,|\supp(m)|}(\shadow_{m}(I))$, and likewise for $J$.
By Lemma \ref{shadowlemma}, we can compare these Betti numbers
with Theorem \ref{HM}.
Applying Peeva's proof \cite{Pe} to our
proof of Theorem \ref{HM}, the graded Betti numbers of
$\shadow_{m}(I)$ and $\shadow_{m}(J)$ differ by
consecutive cancellations.
Thus, there exist integers $c_{i,m}$ such that
$b_{i,m}(J)-b_{i,m}(I)=c_{i,m}+c_{i-1,m}$.

If $m$ has the form $fa^{\ell}b^{s}$, put $n=fa^{s}b^{\ell}$.  Then by
the Mayer-Vietoris sequence and 
Proposition \ref{mayervietoriskernel}, we have 
\[
b_{i,m}(J)+b_{i,n}(J)=b_{i,m}(I)+b_{i,n}(I)+\dim_{\bbbk}(\ker
\Delta_{i,I})_{m} +\dim_{\bbbk}(\ker \Delta_{i-1,I})_{m}.
\]
Set $c_{i,m}=\dim_{\bbbk}(\ker \Delta_{i,I})_{m}$.  

\medskip
Finally, we put 
\[
c_{i,j}=\displaystyle\hspace{-9pt}\sum_{\substack{m=\sigma(m)\\\deg
    m=j}}\hspace{-7pt}c_{i,m}+\hspace{-9pt}\sum_{\substack{
    m=fa^{\ell}b^{s}\\\deg 
    m=j}}\hspace{-7pt}c_{i,m}.
\]
The statement follows from the formula in the proof of Theorem
\ref{shiftbetti}.  
\item[(2)] The statement follows from the proof of Proposition
  \ref{shiftpluspbetti} in the same way that (1) follows from the
  proof of Theorem \ref{shiftbetti}.  Let $m$ be a multidegree.  If $m=fa^{\ell}b^{s+t+1}$ with
  $\ell+t+1=\beta$, put $c_{i,m}=\dim_{\bbbk}(\ker
  \Delta_{i,a^{t+1}I'})_{a^{t+1}m} - \dim_{\bbbk}(\im
  b^{\beta}_{*,i-1,J})_{m}$.  Otherwise, define $c_{i,m}$ as in the
  proof of (1) (making the obvious changes).

Again, we put 
\[
c_{i,j}=\displaystyle\hspace{-18pt}\sum_{\substack{m=fa^{s}b^{s+t+1}\\\deg
    m=j}}\hspace{-16pt}c_{i,m}\;+\;\;\hspace{-18pt}\sum_{\substack{
    m=fa^{\ell}b^{s+t+1}\\\deg 
    m=j}}\hspace{-16pt}c_{i,m}.\qedhere
\]
 \end{itemize}
 \end{proof}

\section{Open problems}
We recall some related problems, and make some
brief remarks about them.

\subsection{The general Lex-Plus-Powers Conjecture}
In Evans' original conjecture, the regular sequence was not required
to consist of monomials:

\begin{conjecture}[Evans, The Lex-Plus-Powers Conjecture]\label{stronglpp}
  Suppose that $F=(f_{1},\cdots, f_{r})$ is any regular sequence
  with $\deg f_{i}=e_{i}$, and define $P=(x_{1}^{e_{1}},\cdots,
  x_{r}^{e_{r}})$.  Let $I$ be any homogeneous ideal containing $F$.
  Then there exists a lex-plus-$P$ ideal $L$ such 
  that $I$ and $L$ have the same Hilbert function. Furthermore, $b_{i,j}(L)\geq
  b_{i,j}(I)$ for all $i,j$.  
\end{conjecture}

A few special cases and reductions are known, due to Francisco,
Richert, and Sabourin \cites{Fr,FR,Ri,RS}, but the conjecture appears
to be wide 
open.   Indeed, the mere existence of the lex-plus-$P$ ideal $L$ is
far from certain; this is the Eisenbud-Green-Harris conjecture
\cites{EGH1,EGH2}. Some special cases are due to Caviglia and Maclagan, Cooper,
 and Richert
\cites{CM,Co,FR
}.  A good survey article on both conjectures is \cite{FR}.

The problem for both conjectures is that the usual first step in
proving Macaulay-type 
theorems is to take an initial ideal, but doing so in this setting
ruins the regular sequence.  Without a monomial ideal, most of our
other techniques are useless.  Unfortunately, Theorem \ref{lpp} does
nothing to resolve this. It does, however, reduce both conjectures to
the same obstacle.  The following statement is equivalent to the
Lex-Plus-Powers conjecture (and, without the last sentence, has been
known for some time to imply the Eisenbud-Green-Harris conjecture):

\begin{conjecture}\label{weaklpp}
  Let $(f_{1},\cdots, f_{r})$ be a regular sequence
  with $deg f_{i}=e_{i}$, and let $P=(x_{1}^{e_{1}},\cdots,
  x_{r}^{e_{r}})$.  Then there exists a monomial ideal $M$ containing
  $P$ such
  that $I$ and $M$ have the same Hilbert function. Furthermore, $M$
  may be chosen so that $b_{i,j}(M)\geq
  b_{i,j}(I)$ for all $i,j$.  
\end{conjecture}

Conjecture \ref{weaklpp} holds if all but one of the $f_{i}$ are pure
powers.  (For example, if $(f_{1},\cdots, f_{r-1})$ =
$(x_{1}^{e_{1}}, \cdots, x_{r-1}^{e_{r-1}})$, then the monomial support of
$f_{r}$ must contain some term not divisible by any of $x_{1}, \cdots,
x_{r-1}$.  Take an initial ideal in some appropriate order, and if
necessary apply Lemma \ref{anymonregseq}.)  Thus, the Lex-Plus-Powers
Conjecture holds for these regular sequences as well:

\begin{proposition}
Let $F$, $I$, and $L$ be as in Conjecture
\ref{stronglpp}, and suppose further that all but one of the $f_{i}$
is a pure power.  Then there exists a lex-plus-$P$ ideal $L$ with the
same Hilbert function as $I$, and, for all $i$ and $j$, we have
$b_{i,j}(L)\geq b_{i,j}(I)$.  
\end{proposition}

In \cite{CM}, Caviglia and Maclagan prove that the
Eisenbud-Green-Harris Conjecture holds whenever
$e_{k}>\sum_{\ell=1}^{k-1}e_{\ell}$ for all $k$.  In light of this
result, we consider the following question potentially tractable:

\begin{problem}
Suppose that, for all $k$, $e_{k}>\sum_{\ell=1}^{k-1}e_{\ell}$.  Does
the Lex-Plus-Powers Conjecture hold for ideals containing a regular
sequence in degrees $(e_{1},\cdots, e_{r})$?
\end{problem}

\subsection{Betti numbers over $S/P$}
Gasharov, Hibi, and Peeva make the following conjecture \cite{GHP}:

\begin{conjecture}
Let $I$ be a homogeneous ideal containing $P$, and let $L$ be the
lex-plus-$P$ ideal with the same Hilbert function.  Let $\bar{I}$ and
$\bar{L}$ be their images in $R=S/P$.  Then the graded Betti numbers of
$\bar{I}$ and $\bar{L}$ satisfy
$b_{i,j}^{R}(\bar{L})\geq b_{i,j}^{R}(\bar{I})$ for all $i$ and $j$.  
\end{conjecture}

This conjecture deals with infinite resolutions.  Nevertheless,
our techniques may give some indication of how to proceed.  For
example, after replacing the Koszul complex with a resolution of $\bbbk$
over $R$, an analog of Corollary \ref{unfixedmulti} continues to
hold.  It is less clear how the rest of the argument might
translate, however.


\end{document}